\UseRawInputEncoding
\documentclass[12pt, reqno]{amsart}
\usepackage[margin=1in]{geometry}
\usepackage{amssymb,latexsym,amsmath,amscd,amsfonts}
\usepackage{latexsym}
\usepackage[mathscr]{eucal}
\usepackage{bm}
\usepackage{mathptmx}

\newcommand{\vsp}{\vspace{5mm}}

\def\textmatrix#1&#2\\#3&#4\\{\bigl({#1 \atop #3}\ {#2 \atop #4}\bigr)}
\def\dispmatrix#1&#2\\#3&#4\\{\left({#1 \atop #3}\ {#2 \atop #4}\right)}
\newcommand{\beg}{\begin{equation}}
	\newcommand{\eeg}{\end{equation}}
\newcommand{\ben}{\begin{eqnarray*}}
	\newcommand{\een}{\end{eqnarray*}}

\newlength{\bibitemsep}
\newlength{\bibparskip}
\let\oldthebibliography\thebibliography
\renewcommand\thebibliography[1]{%
	\oldthebibliography{#1}%
	\setlength{\parskip}{\bibitemsep}%
	\setlength{\itemsep}{\bibparskip}%
}
\newtheorem{thm}{Theorem}[section]

\newtheorem{lem}[thm]{Lemma}

\newtheorem{prop}[thm]{Proposition}
\numberwithin{equation}{section} 
\theoremstyle{definition}
\newtheorem{defn}[thm]{Definition}
\newtheorem{rem}[thm]{Remark}

\newcommand{\HS}{\mathcal H}
\newcommand{\C}{\mathbb{C}}

\newcommand{\D}{\mathbb{D}}

\newcommand{\ov}{\overline}

\begin{document}
	\title[Minimal unitary dilations for commuting contractions]
	{Minimal unitary dilations for commuting contractions}
	
	\author[Pal and Sahasrabuddhe]{Sourav Pal and Prajakta Sahasrabuddhe}
	
	\address[Sourav Pal]{Mathematics Department, Indian Institute of Technology Bombay,
		Powai, Mumbai - 400076, India.} \email{sourav@math.iitb.ac.in}
		
	\address[Prajakta Sahasrabuddhe]{Mathematics Department, Indian Institute of Technology Bombay,
		Powai, Mumbai - 400076, India.} \email{prajakta@math.iitb.ac.in}

	\keywords{Unitary dilation, Minimal dilation space}
	
	\subjclass[2010]{47A15, 47A20, 47B35, 47B38}
	
	\thanks{The first named author is supported by the Seed Grant of IIT Bombay, the CPDA and the MATRICS Award (Award No. MTR/2019/001010) of Science and Engineering Research Board (SERB), India. The second named author has been supported by the Ph.D Fellowship of Council of Scientific and Industrial Research (CSIR), India.}

	\begin{abstract}
		
 For commuting contractions $T_1,\dots ,T_n$ acting on a Hilbert space $\HS$ with $T=\prod_{i=1}^n T_i$, we show that $(T_1, \dots, T_n)$ dilates to commuting isometries $(V_1, \dots , V_n)$ on the minimal isometric dilation space of $T$ with $V=\prod_{i=1}^n V_i$ being the minimal isometric dilation of $T$ if and only if $(T_1^*, \dots , T_n^*)$ dilates to commuting isometries $(Y_1, \dots , Y_n)$ on the minimal isometric dilation space of $T^*$ with $Y=\prod_{i=1}^n Y_i$ being the minimal isometric dilation of $T^*$. Then, we prove an analogue of this result for unitary dilations of $(T_1, \dots , T_n)$ and its adjoint. We find a necessary and sufficient condition such that $(T_1, \dots , T_n)$ possesses a unitary dilation $(W_1, \dots , W_n)$ on the minimal unitary dilation space of $T$ with $W=\prod_{i=1}^n W_i$ being the minimal unitary dilation of $T$. We show an explicit construction of such a unitary dilation on both Sch$\ddot{a}$ffer and Sz. Nagy-Foias minimal unitary dilation spaces of $T$. Also, we show that a relatively weaker hypothesis is necessary and sufficient for the existence of such a unitary dilation when $T$ is a $C._0$ contraction, i.e. when ${T^*}^n \rightarrow 0$ strongly as $n \rightarrow \infty $. We construct a different unitary dilation for $(T_1, \dots , T_n)$ when $T$ is a $C._0$ contraction.
 
	\end{abstract}
	
	\maketitle
	
	\tableofcontents
		
	\section{Introduction}
	\vspace{0.4cm}
	
Throughout the paper all operators are bounded linear operators acting on complex Hilbert spaces. A contraction is an operator with norm not greater than $1$. We begin with the definitions of isometric and unitary dilations of a tuple of commuting contractions.

\begin{defn}
Let $(T_1, \dots , T_n)$ be a tuple of commuting contractions acting on a Hilbert space $\HS$. A commuting tuple of unitaries $(W_1, \dots , W_n)$ acting on a Hilbert space $\mathcal K'$ is said to be a \textit{unitary dilation} of $(T_1, \dots , T_n)$ if $\HS$ can be realized as a closed linear subspace of $\mathcal K'$ and for any non-negative integers $k_1, \dots , k_n$ we have
\[
T_1^{k_1} \dots T_n^{k_n}=P_{\HS} (W_1^{k_1}\dots W_n^{k_n})|_{\HS},
\]	
where $P_{\HS}:\mathcal K' \rightarrow \HS$ is the orthogonal projection. Moreover, such a unitary dilation is called \textit{minimal} if
\[
\mathcal K' = \ov{Span}\; \{ W_1^{t_1}\dots W_n^{t_n}h\;:\; h\in \HS , \; t_1, \dots , t_n \in \mathbb Z \}.
\]
Similarly, a commuting tuple of isometries $(V_1, \dots , V_n)$ acting on a Hilbert space $\mathcal K$ is said to be an \textit{isometric dilation} of $(T_1, \dots , T_n)$ if $\HS$ can be realized as a closed linear subspace of $\mathcal K$ and for any non-negative integers $k_1, \dots , k_n$ we have
\[
T_1^{k_1} \dots T_n^{k_n}=P_{\HS} (V_1^{k_1}\dots V_n^{k_n})|_{\HS}.
\]
Moreover, such an isometric dilation is called \textit{minimal} if
\[
\mathcal K = \ov{Span}\; \{ V_1^{k_1}\dots V_n^{k_n}h\;:\; h\in \HS , \; k_1, \dots , k_n \in \mathbb N \cup \{0\} \}.
\]
\end{defn}
A path-breaking work due to Sz. Nagy, \cite{Nagy 1} established that a contraction always dilates to a unitary. The result of Sz. Nagy was generalized by Ando, \cite{Ando} to a pair of commuting contractions. A pair of commuting contractions possesses an unconditional unitary dilation. Later, Parrott showed by a counter example that a triple of commuting contractions may not dilate to a triple of commuting unitaries, see \cite{Par}. In a more general operator theoretic language rational dilation succeeds on the closed unit disk $\ov{\mathbb D}$ and on the closed bidisc $\ov{\mathbb D^2}$ and fails on the closed polydisc $\ov{\D^n}$ when $n \geq 3$. Thus, for $n\geq 3$ history witnessed only conditional unitary dilations for commuting contractions. Some notable works in this direction are due to Agler \cite{Agler 1}, Brehmer \cite{Bre}, Curto, Vasilescu \cite{Cur:Vas 1, Cur:Vas 2}, Ball, Li, Timotin, Trent \cite{Bal:Tim:Tre}, Ball, Trent, Vinnikov \cite{Bal:Tre:Vin}, Eschmeier \cite{Esch-1}, M$\ddot{u}$ller, Vasilescu \cite{Mul:Vas}, Burdak \cite{Bur 1}, Barik, Das, Sarkar \cite{B:D:H:S} and many others, see the reference list and the references therein.\\

In this article we study isometric and unitary dilations of a tuple of commuting contractions $(T_1, \dots , T_n)$ on the minimal isometric and minimal unitary dilation spaces (which are always unique upto unitaries) of $T=\prod_{i=1}^n T_i$. Our first main result is the following.

\begin{thm} \label{intro:mainlemma}
Let $T_1,\ldots T_n\in \mathcal{B}(\mathcal{H})$ be commuting contractions and let $\mathcal K$ be the minimal isometric dilation space for their product $T=\Pi_{i=1}^nT_i$. Then $(T_1, \dots , T_n)$ possesses an isometric dilation $(V_1,\ldots ,V_n)$ on $\mathcal K$ with $V=\prod_{i=1}^{n}V_i$ being the minimal isometric dilation of $T$ if and only if $(T_1^*, \dots , T_n^*)$ possesses an isometric dilation $(Y_1,\ldots ,Y_n)$ on $\mathcal K_*$ with $Y=\prod_{i=1}^{n}Y_i$ being the minimal isometric dilation of $T$, where $\mathcal K_*$ is the minimal isometric dilation space for $T^*$.
\end{thm}
   
In our second main result, which is stated below, we find a necessary and sufficient condition such that $(T_1, \dots , T_n)$ dilates to a commuting unitary tuple $(W_1, \dots , W_n)$ on the minimal unitary dilation space $\mathcal K'$ of $T$ with $\prod_{i=1}^n W_i=W$ being the minimal unitary dilation of $T$. 

\begin{thm} \label{intro:uni-main}
			
	Let $T_1,\ldots,T_n\in \mathcal{B}(\mathcal{H})$ be commuting contractions, $T=\Pi_{j=1}^nT_j$ and $T_i'=\Pi_{i\neq j} T_j$ for $1\leq i \leq n$. 
	\begin{itemize}
		\item[(a)] If $\widetilde{\mathcal{K}}_0$ is a minimal unitary dilation space for $T$, then $(T_1,\ldots ,T_n)$ possesses a unitary dilation $(W_1,\ldots,W_{n})$ on $\widetilde{\mathcal{K}}_0$ with $W=\prod_{i=1}^{n}W_i$ being the minimal unitary dilation of $T$ if and only if there exist unique projections $P_1,\ldots ,P_n$ and unique commuting unitaries $U_1,\ldots ,U_n$ in $\mathcal{B}(\mathcal{D}_T)$ with $\prod_{i=1}^n U_i=I$ such that the following hold for $i=1, \dots, n:$
		\begin{enumerate} 
			\item $D_TT_i=P_i^{\perp}U_i^*D_T+P_iU_i^*D_TT$ ,
			\item  $P_i^{\perp}U_i^*P_j^{\perp}U_j^*=P_j^{\perp}U_j^*P_i^{\perp}U_i^*$ ,
			\item $U_iP_iU_jP_j=U_jP_jU_iP_i$ ,
			\item $D_TU_iP_iU_i^*D_T=D_{T_i}^2$ ,
			\item  $P_1+U_1^*P_2U_1+U_1^*U_2^*P_3U_2U_1+\ldots +U_1^*U_2^*\ldots U_{n-1}^*P_nU_{n-1}\ldots U_2U_1 =I_{\mathcal{D}_{T}}$.
			
		\end{enumerate} 
				
		\item[(b)] Such a unitary dilation is minimal and unique in the sense that if $(X_1, \dots , X_n)$ on $\widetilde{\mathcal K}_1$ is another unitary dilation of $(T_1, \dots , T_n)$ such that the product $X=\prod_{i=1}^nX_i$ is a minimal unitary dilation of $T$, then there is a unitary $\rho:\widetilde{\mathcal K}_0 \rightarrow \widetilde{\mathcal K}_1$ such that $(X_1, \dots , X_n)=(\rho* W_1 \rho, \dots , \rho* W_n \rho)$.
		
	\end{itemize}
	
\end{thm}

This is Theorem \ref{Unimain} in this paper. We show an explicit construction of such a unitary dilation on the Sch$\ddot{a}$ffer's minimal unitary dilation space $\widetilde{\mathcal K_0}= l^2(\mathcal D_T)\oplus \HS \oplus l^2(\mathcal D_{T^*})$ of $T$, where $\mathcal D_T=\ov{Ran}(I-T^*T)^{\frac{1}{2}}$ and $\mathcal D_{T^*}=\ov{Ran} (I-TT^*)^{\frac{1}{2}}$. Note that upto a unitary $\widetilde{\mathcal K_0}$ is the smallest Hilbert space on which $(T_1, \dots , T_n)$ can have such a unitary dilation and the reason is that it is the minimal unitary dilation space of the product $\prod_{i=1}^nT_i=T$.

In \cite{Sou:Pra}, the authors of this article show that $(T_1, \dots , T_n)$ dilates to a commuting tuple of isometries $(V_1, \dots , V_n)$ on the minimal isometric dilation space $\mathcal K$ of $T$ with $V=\prod_{i=1}^n V_i$ being the minimal isometric dilation of $T$ if and only if the conditions $(1)-(5)$ of Theorem \ref{intro:uni-main} hold. Naturally, any unitary extension of the subnormal tuple $(V_1, \dots , V_n)$ becomes a unitary dilation for $(T_1, \dots , T_n)$. The most interesting fact about the unitary dilation in Theorem \ref{intro:uni-main} is that without any additional hypothesis $(V_1, \dots , V_n)$ on $\mathcal K$ admits a minimal unitary extension $(W_1, \dots , W_n)$ on $\mathcal K'$, the minimal unitary dilation space for $T$.
 
  We also show in Theorem \ref{coromain} that such a unitary dilation can be constructed with the conditions $(1)-(4)$ only, though we do not have an exact converse part then. We construct a special unitary dilation for $(T_1, \dots , T_n)$ when the product $T$ is a $C._0$ contraction, i.e. ${T^*}^n\rightarrow 0$ strongly as $n\rightarrow \infty$. Indeed, in Theorem \ref{purediluni} with an explicit construction we show that an analogue of Theorem \ref{intro:uni-main} can be obtained with a weaker hypothesis when $T$ is a $C._0$ contraction. This is another main result of this paper. The other achievement of this paper is Theorem \ref{thm:Nagy-Foias} in which we construct an explicit unitary dilation for $(T_1, \dots , T_n)$ on the Sz. Nagy-Foias minimal unitary dilation space of $T$ when $T$ is a completely non-unitary (c.n.u.) contraction, i.e. when $T$ is missing a unitary summand in its canonical decomposition. We accumulate a few preparatory results in Section \ref{sec:02}.
  
	
	\section{A few preparatory results} \label{sec:02}
	
	\vspace{0.4cm}
	
\noindent We begin with a famous result due to Douglas, Muhly and Pearcy. This result will play a major role in the proof of the main results of this paper.
   
	\begin{prop} [\cite{Dou:Muh:Pea}, Proposition 2.2] \label{prop:21}	

For $i=1,2,$ let $T_i$ be a contraction on a Hilbert space $\HS_i$ and let $X$ be an operator from $\HS_2$ to $\HS_1$. A necessary and sufficient condition that the operator on $\HS_1 \oplus \HS_2$ defined by the matrix $\begin{bmatrix}
T_1 & X \\
0 & T_2
\end{bmatrix}
$
be a contraction is that there exists a contraction $C$ mapping $\HS_2$ into $\HS_1$ such that
$
X= (I-T_1T_1^*)^{\frac{1}{2}}C(I-T_2^*T_2)^{\frac{1}{2}}.
$

\end{prop}

The following theorem is another general and useful	result in operator theory.
	 
	\begin{thm}[\cite{Bha:Sau}, Lemma 13] \label{forunidil}
		Let $(R_1,\dots , R_{n-1},U)$ on $\mathcal{K}$ be a dilation of $(S_1,\dots , S_{n-1},P)$ on $\mathcal{H}$, where $P$ is a contraction on $\mathcal{H}$ and $U$ on $\mathcal{K}$ is the Schäffer minimal unitary
		dilation of $P$. Then for all $j=1,2,\ldots,n-1$, $R_j$ admits a matrix representation of the form
		\[\begin{bmatrix}
			*&*&*\\
			0&S_j&*\\
			0&0&*
		\end{bmatrix}
		\]	with respect to the decomposition $\mathcal{K}=l^2(\mathcal{D}_T)\oplus \mathcal{H}\oplus l^2(\mathcal{D}_{T^*})$.
	\end{thm}
	
	In \cite{Ber}, Berger, Coburn and Lebow found the most popular factorization of a pure isometry. This is stated as Theorem \ref{BCL} in this paper. Later Bercovici, Douglas and Foias proved a finer version of that result which we are going to state below.
		
			\begin{lem}[Bercovici, Douglas and Foias, \cite{Berc:Dou:Foi}] \label{BDF lemma O}
			Let $U_1,\ldots ,U_n$ be unitaries on Hilbert space $\mathcal{H}$  and $P_1,\ldots ,P_n$ be orthogonal projections on $\mathcal{H}$. For $1\leq i \leq n$,  let $V_i=M_{U_iP_i^{\perp}+zU_iP_i}$. Then $(V_1,\ldots ,V_n)$ defines a commuting $n$-tuple of isometries with $\Pi_{i=1}^nV_i=M_z$ if and only if the following conditions are satisfied. \begin{enumerate}
				\item $U_iU_j=U_jU_i$ for all $1\leq i <j \leq n$,
				\item $U_1\ldots U_n=I_{\mathcal{H}}$,
				\item $P_j+U_j^*P_iU_j=P_i+U_i^*P_jU_i \leq I_{\mathcal{H}}$, for all $i\neq j$ and 
				\item  $P_1+U_1^*P_2U_1+U_1^*U_2^*P_3U_2U_1+\ldots +U_1^*U_2^*\ldots U_{n-1}^*P_nU_{n-1}\ldots U_2U_1 =I_{\mathcal{H}}$.
				
			\end{enumerate} 
		\end{lem}
		The following result also appeared in \cite{Berc:Dou:Foi}.  
		\begin{lem}[\cite{Berc:Dou:Foi}, Lemma 2.2]\label{BDF lemma 1}
			Consider unitary operators $U,U_1,U_2$ and orthogonal projections $P,P_1$ and $P_2$ on Hilbert space $\mathcal{H}$. If $V_{U,P},V_{U_1,P_1}$ and $V_{U_2,P_2}$ on $H^2(\mathcal{H})$ are defined as $V_{U,P}=M_{P^{\perp}U^*+zPU^*}$, $V_{U_1,P_1}=M_{P_1^{\perp}U_1^*+zP_1U_1^*}$ and $V_{U_2,P_2}=M_{P_2^{\perp}U_2^*+zP_2U_2^*}$, then the following are equivalent. \begin{enumerate}
				\item[(i)] $ V_{U,P}=V_{U_1,P_1}V_{U_2,P_2} $,
				\item[(ii)] $U=U_1U_2  $ and $P=P_1+U_1^*P_2U_1$.
			\end{enumerate}
		\end{lem}
		
		Let us state a pair of lemma that are analogous to Lemmas \ref{BDF lemma O} \& \ref{BDF lemma 1} respectively. In fact they provide factorization of similar kind for co-analytic symbols. The results can be proved using similar techniques as in the proof of Lemmas \ref{BDF lemma O} \& \ref{BDF lemma 1} in \cite{Berc:Dou:Foi} and thus we skip the proofs. These results will be used in the proof of the main unitary dilation theorem.

\begin{lem} \label{BDF lemma alt}
	Let $U_1,\ldots ,U_n$ be unitaries and $P_1,\ldots ,P_n$ be orthogonal projections on $\mathcal{H}$. For each $1\leq i \leq n$,  let $V_i=M_{U_iQ_i^{\perp}+\bar{z}U_iP_i}$. Then $(V_1,\ldots ,V_n)$ defines a commuting $n$-tuple of co-isometries with $\Pi_{i=1}^nV_i=M_{\bar{z}}$ if and only if the following conditions are satisfied. \begin{enumerate}
		\item $U_iU_j=U_jU_i$ for all $1\leq i <j \leq n$,
		\item $U_1\ldots U_n=I_{\mathcal{H}}$,
		\item $P_j+U_j^*P_iU_j=P_i+U_i^*P_jU_i \leq I_{\mathcal{H}}$, for all $i\neq j$ and 
		\item  $P_1+U_1^*P_2U_1+U_1^*U_2^*P_3U_2U_1+\ldots +U_1^*U_2^*\ldots U_{n-1}^*P_nU_{n-1}\ldots U_2U_1 =I_{\mathcal{H}}$.
	\end{enumerate} 
\end{lem}

\begin{lem} \label{BDF lemma 2}
	Consider unitary operators $U,U_1,U_2$ and orthogonal projections $P,P_1$ and $P_2$ on Hilbert space $\mathcal{H}$. If $W_{U,P},W_{U_1,P_1}$ and $W_{U_2,P_2}$ are Toeplitz operators on $H^2(\mathcal{H})$ defined by $W_{U,P}=M_{UP^{\perp}+\bar{z}UP}$, $V_{U_1,P_1}=M_{U_1Q_1^{\perp}+\bar{z}U_1P_1}$ and $V_{U_2,P_2}=M_{U_2P_2^{\perp}+\bar{z}U_2P_2}$, then the following are equivalent. 
	\begin{enumerate}
		\item[(i)] $ V_{U,P}=V_{U_1,P_1}V_{U_2,P_2} $,
		\item[(ii)] $U=U_1U_2 $ and $P=P_2+U_2^*P_1U_2$.
	\end{enumerate}
	
\end{lem}

	\vspace{0.2cm}		

\section{Sch$\ddot{a}$ffer-type minimal unitary dilation}	

\vspace{0.4cm}

\noindent In this Section, we find a necessary and sufficient condition such that a tuple of commuting contraction $(T_1, \dots , T_n)$ dilates to a tuple of commuting unitaries acting on the minimal unitary dilation space of $T=\prod_{i=1}^nT_i$. The dilation is minimal and we also show an explicit construction of such a minimal unitary dilation on the Sch$\ddot{a}$ffer's minimal space. In \cite{Sou:Pra}, we have constructed an isometric dilation in such scenario. We recall that result here. This will be frequently used throughout the paper. 
		
	\begin{thm}[\cite{Sou:Pra}, Theorem 3.4] \label{main}
			Let $T_1,\ldots,T_n\in \mathcal{B}(\mathcal{H})$ be commuting contractions, $T={ \prod_{i=1}^n T_i}$ and $T_i'={ \prod_{i\neq j} T_j}$ for $1\leq i \leq n$. 
	\begin{itemize}
		\item[(a)] If $\mathcal{K}$ is the minimal isometric dilation space of $T$, then $(T_1,\ldots ,T_n)$ possesses an isometric dilation $(V_1,\ldots,V_{n})$ on $\mathcal{K}$ with $V=\Pi_{i=1}^nV_i$ being the minimal isometric dilation of $T$ if and only if there are unique orthogonal projections $P_1,\ldots ,P_n$ and unique commuting unitaries $U_1,\ldots ,U_n$ in $\mathcal{B}(\mathcal{D}_T)$ with $\prod_{i=1}^n U_i=I_{\mathcal{D}_{T}}$ such that the following conditions are satisfied for each $i=1, \dots, n$:
		\begin{enumerate} 
			\item $D_TT_i=P_i^{\perp}U_i^*D_T+P_iU_i^*D_TT$ , 
			\item  $P_i^{\perp}U_i^*P_j^{\perp}U_j^*=P_j^{\perp}U_j^*P_i^{\perp}U_i^*$ ,
			\item $U_iP_iU_jP_j=U_jP_jU_iP_i$ ,
			\item $D_TU_iP_iU_i^*D_T=D_{T_i}^2$ ,
			\item  $P_1+U_1^*P_2U_1+U_1^*U_2^*P_3U_2U_1+\ldots +U_1^*U_2^*\ldots U_{n-1}^*P_nU_{n-1}\ldots U_2U_1 =I_{\mathcal{D}_T}$.
		\end{enumerate}
		\vspace{2mm}

		\item[(b)] Such an isometric dilation is minimal and unique in the sense that if $(W_1, \dots , W_n)$ on $\mathcal K_1$ and $(Y_1, \dots , Y_n)$ on $\mathcal K_2$ are two isometric dilations of $(T_1, \dots , T_n)$ such that $W=\prod_{i=1}^nW_i$ and $Y=\prod_{i=1}^nY_i$ are minimal isometric dilations of $T$ on $\mathcal K_1$ and $\mathcal K_2$ respectively, then there is a unitary $\widetilde{U}:\mathcal K_1 \rightarrow \mathcal K_2$ such that $(W_1, \dots , W_n)=(\widetilde{U}^*Y_1\widetilde{U}, \dots , \widetilde{U}^*Y_n\widetilde{U})$.
	\end{itemize}			
			
\end{thm}

We have from \cite{Sou:Pra} an existence of an isometric dilation on the minimal dilation space of $T$ with an assumption of four out of five conditions of Theorem \ref{main}. This is a weaker version of Theorem \ref{main} where we miss an appropriate converse part. 

\begin{thm} [\cite{Sou:Pra}, Theorem 3.5] \label{coro-main}
	Let $T_1,\ldots, T_n \in \mathcal{B}(\mathcal{H})$ be commuting contractions, $T_i'=\prod_{i\neq j} T_j$ for all $1\leq i \leq n$ and $T=\prod_{i=1}^nT_i$. Then $(T_1,\ldots ,T_n)$ possesses an isometric dilation on the minimal isometric dilation space of $T$, if there are projections $P_1,\ldots ,P_n$ and commuting unitaries $U_1,\ldots ,U_n$ in $\mathcal{B}(\mathcal{D}_T)$ such that the following conditions hold for $i=1, \dots, n$:
	\begin{enumerate} 
		\item $D_TT_i=P_i^{\perp}U_i^*D_T+P_iU_i^*D_TT$ ,
		\item  $P_i^{\perp}U_i^*P_j^{\perp}U_j^*=P_j^{\perp}U_j^*P_i^{\perp}U_i^*$ ,
		\item $U_iP_iU_jP_j=U_jP_jU_iP_i$ ,
		\item $D_TU_iP_iU_i^*D_T=D_{T_i}^2$.
	\end{enumerate} 
	Conversely, if $(T_1,\ldots ,T_n)$ possesses an isometric dilation $(\widehat{V}_1,\ldots,\widehat{V}_{n})$ with $V=\prod_{i=1}^nV_i$ being the minimal isometric dilation of $T$, then there are unique projections $P_1,\ldots ,P_n$ and unique commuting unitaries $U_1,\ldots ,U_n$ in $\mathcal{B}(\mathcal{D}_T)$ satisfying the conditions $(1)-(4)$ above.
\end{thm}

The next few lemmas are useful for proving the main results of this Section. Note that the next lemma is proved in \cite{Bha-01} in a more general setting.

\begin{lem}\label{fundexist}
	Let $\HS$ be a Hilbert space and let $A,B\in \mathcal{B}(\mathcal{H})$ be commuting contractions. Then there is a unique contraction $C\in \mathcal{B}(\mathcal{D}_{AB})$ such that $D_B^2A=D_{AB}CD_{AB}$. 
\end{lem}
\begin{proof}
	Let us consider the operator $Y$ on $\HS \oplus \HS$ defined by the matrix $ Y=\begin{bmatrix}
		A^*B^*&D_{B}^2A\\
		0&AB
	\end{bmatrix}
	$. Since $\begin{bmatrix}
		B^*&D_B^2\\
		0&B
	\end{bmatrix}$
	is a contraction by Proposition \ref{prop:21} and since $A$ is a contraction, it follows that
	\[
	Y= \begin{bmatrix}
		A^*B^*&D_{B}^2A\\
		0&AB
	\end{bmatrix}
	 = \begin{bmatrix}
		B^*&D_B^2\\
		0&B
	\end{bmatrix} \begin{bmatrix}
		A^*&0\\
		0&A
	\end{bmatrix}
	\]
	is a contraction. So, again by Proposition \ref{prop:21}, there is a contraction $F\in \mathcal{B}(\mathcal{H})$ such that
	\[
	D_{B}^2A=D_{AB}FD_{AB}.
	\]
	Suppose the matrix of $F$ with respect to the decomposition $\mathcal{H} = \mathcal{D}_{AB}\oplus ker(D_{AB}) $ is $\begin{bmatrix}
		F_{11}&F_{12}\\
		F_{21}&F_{22}
	\end{bmatrix}$. Since $D_B^2A=D_{AB}FD_{AB}$, it follows that $D_B^2A$ vanishes on $ker(D_{AB})$. Therefore, the matrix of $D_{B}^2A$ with respect to decomposition $\HS=\mathcal{D}_{AB}\oplus ker(D_{AB})$ is $\begin{bmatrix}
		D_B^2A&0\\
		0&0
	\end{bmatrix}$. So, we have $D_B^2A=D_{AB}F_{11}D_{AB}$, where $F_{11} \in \mathcal{B}(\mathcal{D}_{AB})$ is a contraction. 
	
	For the uniqueness part, let there be two contractions $C, C_1 \in \mathcal B(\mathcal D_{AB})$ satisfying $D_B^2A=D_{AB}XD_{AB}$. Then $D_{AB}(C-C')D_{AB}=0$ and consequently $\langle (C-C')D_{AB}h,D_{AB}h\rangle=0$ for any $h\in \HS$. This shows that $C=C_1$ and the proof is complete.
	  
\end{proof}

The following result can be found in the literature (e.g. \cite{Ball-Sau1}), though here we include a shorter proof for the sake of completeness. 

	\begin{lem}\label{Fundamental equations}
		Let $T_1,\ldots , T_n\in \mathcal{B}(\mathcal{H})$ be commuting contractions. Let $T=\Pi_{j=1}^nT_j$, and $T_i'=\Pi_{i\neq j}T_j$ for $1 \leq i \leq n$. Then
		$
		D_TT_i=F_iD_T+F_i'^*D_TT \text{ and } D_{T}T_i'=F_i'D_T+F_i^*D_TT ,
		$
		where $F_i,F_i' \in \mathcal{B}(\mathcal{D}_T)$ are unique solutions of $D_{T_i'}^2T_i=D_TX_iD_T$ and $D_{T_i}^2T_i'=D_TX_{i}'D_T$ respectively for each $i$.
	\end{lem}
	
	\begin{proof}
		Let $G=F_i'^*D_TT+F_iD_T-D_TT_i$. Then $G$ is defined from $\mathcal{H}\to \mathcal{D}_T$. Since $F_i,F_i'$ are unique solutions of $D_{T_i'}^2T_i=D_TX_iD_T$ and $D_{T_i}^2T_i'=D_TX_{i}'D_T$ respectively for each $i$, we have that
		\[
			D_TG =D_TF_i'^*D_TT+D_TF_iD_T-D_T^2T_i =(T_i'^*-T^*T_i)T+(T_i-T_i'^*T)-T_i+T^*TT_i=0.			
		\]
	Now for any $h,h'\in \mathcal{H}$, $ \langle Gh,D_Th' \rangle = \langle D_TGh,h' \rangle =0$.	Hence $G=0$ and consequently $D_TT_i=F_iD_T+F_i'^*D_TT$. We can have a similar proof for $D_TT_i'=F_i'D_T+F_i^*D_TT$.
	
	\end{proof}
	
\noindent \textbf{\textit{A refined Berger-Coburn-Lebow Model Theorem.}} In \cite{Ber}, Berger, Coburn and Lebow found the following remarkable factorization of a pure isometry. In Theorem 3.10 in \cite{S.Pal1}, the first named author of this paper make a slight refinement of that result in the following manner.

\begin{thm} [Berger-Coburn-Lebow, \cite{Ber}] \label{BCL}

Let $V_1, \dots , V_n$ be commuting isometries on $\HS$ such that $V=\prod_{i=1}^n V_i$ is a pure isometry. Let $V_i'=\prod_{j\neq i} V_j$ for $1\leq i \leq n$. Then, there exist projections $P_1, \dots , P_n$ and unitaries $U_1, \dots , U_n$ in $\mathcal B(\mathcal D_{V^*})$ such that
\[
(V_1, \dots , V_n , V) \equiv (T_{P_1^{\perp}U_1+ zP_1U_1 }, \dots , T_{P_n^{\perp}U_n+zP_nU_n }, T_z)  \;\; \text{ on } \; \; H^2(\mathcal D_{V^*}).
\] 
Moreover, $U_i^*P_i^{\perp} , P_iU_i$ are unique operators such that $V_i^*- V_i'V^*=D_{V^*}U_i^*P_i^{\perp}D_{V^*}$ and $V_i'^*- V_iV^*=D_{V^*}P_iU_iD_{V^*}$ respectively for each $i=1, \dots , n$.

\end{thm}

We state a lemma below and its proof will take cues from the proof of Theorem \ref{main} from \cite{Sou:Pra}. This will be useful.

\begin{lem} \label{lem:main-01}

	Let $T_1,\dots ,T_n\in \mathcal{B}(\mathcal{H})$ be commuting contractions, $T=\prod_{i=1}^nT_i$ and $T_i'=\prod_{j \neq i}T_j$ for $1\leq i \leq n$. Let $\mathcal K$ be the minimal isometric dilation space of $T$. If $(T_1,\dots ,T_n)$ possesses an isometric dilation $(V_1,\dots ,V_n)$ on $\mathcal{K}$ with $V=\prod_{i=1}^nV_i$ being the minimal isometric dilation of $T$, then the following hold:
	\begin{enumerate}
		\item[(1)] $D_TT_i'=U_iP_iD_T+U_iP_i^{\perp}D_TT$,
		\item[(2)] $D_TP_i^{\perp}D_T=D_{T_i'}^2$.
	\end{enumerate}
	
\end{lem}

\begin{proof}
	Since $(T_1,\dots ,T_n)$ possesses an isometric dilation $(V_1,\dots ,V_n)$ on $\mathcal{K}$ with $V=\prod_{i=1}^nV_i$ being the minimal isometric dilation of $T$, by part-$(b)$ of Theorem \ref{main}, we can assume without loss of generality that
	
	\[	 V_i=\begin{bmatrix}
		T_i&0&0&0&\ldots \\
		P_iU_i^*D_T&P_i^{\perp}U_i^*&0&0&\ldots \\
		0&P_iU_i^*&P_i^{\perp}U_i^*&0&\ldots \\
		0&0&P_iU_i^*&P_i^{\perp}U_i^*&\ldots \\
		\ldots & \ldots & \ldots & \ldots & \ldots
	\end{bmatrix}  \text{ and } V= \begin{bmatrix}
		T&0&0&0&\ldots \\
		D_T&0&0&0&\ldots \\
		0&I&0&0&\ldots \\
		0&0&I&0&\ldots \\
		\ldots & \ldots & \ldots & \ldots & \ldots
	\end{bmatrix}, \; \text{for }1\leq i\leq n,
	\]
 following the construction of the isometric dilation as in Theorem \ref{main} from \cite{Sou:Pra}.  Here $P_1,\dots ,P_n$ are unique projections and $U_1,\dots ,U_n$ are unique commuting unitaries on $\mathcal{D}_T$ such that $\prod_{i=1}^nU_i=I$ and the conditions $(1)-(5)$ of the Theorem \ref{main} are satisfied. Let $V_i'=\prod_{j\neq i}V_j$ for any $i=1,\dots ,n$. Since $V_1,\dots ,V_n$ and $V=\prod_{i=1}^nV_i$ are isometries, we have $V_i'=V_i^*V$. Hence for any $1\leq i \leq n$, we obtain
 \begin{align*}
 V_i'=V_i^*V&=\begin{bmatrix}
 	T_i^*&D_TU_iP_i&0&0&\ldots \\
 	0&U_iP_i^{\perp}&U_iP_i&0&\ldots \\
 	0&0&U_iP_i^{\perp}&U_iP_i&\ldots \\
 	0&0&0&U_iP_i^{\perp}&\ldots \\
 	\ldots & \ldots & \ldots & \ldots & \ldots
 \end{bmatrix}\begin{bmatrix}
 T&0&0&0&\ldots \\
 D_T&0&0&0&\ldots \\
 0&I&0&0&\ldots \\
 0&0&I&0&\ldots \\
 \ldots & \ldots & \ldots & \ldots & \ldots
\end{bmatrix}\\
&=\begin{bmatrix}
	T_i^*T_i+D_TU_iP_iD_T&0&0&0&\ldots \\
	U_iP_i^{\perp}D_T&U_iP_i&0&0&\ldots \\
	0&U_iP_i^{\perp}&U_iP_i&0&\ldots \\
	0&0&U_iP_i^{\perp}&U_iP_i&\ldots \\
	\ldots & \ldots & \ldots & \ldots & \ldots
\end{bmatrix}\\
&=\begin{bmatrix}
T_i'&0&0&0&\ldots \\
	U_iP_i^{\perp}D_T&U_iP_i&0&0&\ldots \\
	0&U_iP_i^{\perp}&U_iP_i&0&\ldots \\
	0&0&U_iP_i^{\perp}&U_iP_i&\ldots \\
	\ldots & \ldots & \ldots & \ldots & \ldots
\end{bmatrix},
 \end{align*} where the equation $ T_i^*T_i+D_TU_iP_iD_T=T_i'$ follows from $(a')$ in the $(\Rightarrow)$ part of the proof of Theorem \ref{main} from \cite{Sou:Pra}. Thus, $(2,1)$ entry of $V_i'V=VV_i' $ gives us that $D_TT_i'=U_iP_iD_T+U_iP_i^{\perp}D_TT$ for any $i=1,\dots ,n$.
Since $V_i'$ is an isometry, the $(1,1)$ entry of $V_i'^*V_i'=I$ gives $D_TP_i^{\perp}D_T=I-T_i'^*T_i'=D_{T_i'}^2$. 
	 
\end{proof}

We now show that the existence of an isometric dilation for a tuple of commuting contractions $(T_1, \dots , T_n)$ on the minimal isometric dilation space of their product $T=\prod_{i=1}^nT_i$ is equivalent to the existence of a minimal isometric dilation for $(T_1^*, \dots, T_n^*)$ on the minimal isometric dilation space of $T^*$. This is one of the main results of this article.

\begin{thm} \label{mainlemma}
Let $T_1,\ldots ,T_n\in \mathcal{B}(\mathcal{H})$ be commuting contractions and let $\mathcal K$ be the minimal isometric dilation space for their product $T=\Pi_{i=1}^nT_i$. Then $(T_1, \dots , T_n)$ possesses an isometric dilation $(V_1,\ldots ,V_n)$ on $\mathcal K$ with $V=\prod_{i=1}^{n}V_i$ being the minimal isometric dilation of $T$ if and only if $(T_1^*, \dots , T_n^*)$ possesses an isometric dilation $(Y_1,\ldots ,Y_n)$ on $\mathcal K_*$ with $Y=\prod_{i=1}^{n}Y_i$ being the minimal isometric dilation of $T^*$, where $\mathcal K_*$ is the minimal isometric dilation space for $T^*$.

\end{thm}
\begin{proof}

It suffices to prove one side of the theorem, because, the other side follows by an analogous argument. So, we assume that $(T_1, \dots , T_n)$ possesses an isometric dilation $(V_1, \dots , V_n)$ on $\mathcal K$ with $V=\prod_{i=1}^{n}V_i$ being the minimal isometric dilation of $T$. Then by Theorem \ref{main}, there are unique projections $P_1,\ldots ,P_n$ and unique commuting unitaries $U_1,\ldots ,U_n$ in $\mathcal{B}(\mathcal{D}_T)$ with $\prod_{i=1}^nU_i=I$ such that the following conditions are satisfied for $1\leq i \leq n$:   
	\begin{enumerate} 
	\item $D_TT_i=P_i^{\perp}U_i^*D_T+P_iU_i^*D_TT$ ,
    \item  $P_i^{\perp}U_i^*P_j^{\perp}U_j^*=P_j^{\perp}U_j^*P_i^{\perp}U_i^*$ ,
    \item $U_iP_iU_jP_j=U_jP_jU_iP_i$ ,
    \item $D_TU_iP_iU_i^*D_T=D_{T_i}^2$ ,
    \item  $P_1+U_1^*P_2U_1+U_1^*U_2^*P_3U_2U_1+\ldots +U_1^*U_2^*\ldots U_{n-1}^*P_nU_{n-1}\ldots U_2U_1 =I_{\mathcal{D}_{T}}$.
	\end{enumerate}
	To prove that $(T_1^*, \dots , T_n^*)$ possesses an isometric dilation $(Y_1, \dots , Y_n)$ on $\mathcal K_*$ with $Y=\prod_{i=1}^n Y_i$ being the minimal isometric dilation of $T^*$. By Theorem \ref{main} it suffices if we show the existence of unique projections $Q_1,\ldots ,Q_n$ and unique commuting unitaries $\widetilde{U}_1,\ldots ,\widetilde{U}_n$ in $\mathcal{B}(\mathcal{D}_{T^*})$ with $\prod_{i=1}^n\widetilde{U}_i=I$ satisfying    
	\begin{enumerate} 
		\item[$(1')$] $D_{T^*}T_i^*=Q_i^{\perp}\widetilde{U}_i^*D_{T^*}+Q_i\widetilde{U}_i^*D_{T^*}T^*$ ,
		\item[$(2')$] $Q_i^{\perp}\widetilde{U}_i^*Q_j^{\perp}\widetilde{U}_j^*=Q_j^{\perp}\widetilde{U}_j^*Q_i^{\perp}\widetilde{U}_i^*$ ,
		\item[$(3')$] $\widetilde{U}_iQ_i\widetilde{U}_jQ_j=\widetilde{U}_jQ_j\widetilde{U}_iQ_i$ ,
		\item[$(4')$] $D_{T^*}\widetilde{U}_iQ_i\widetilde{U}_i^*D_{T^*}=D_{T_i^*}^2$ ,
		\item[$(5')$] $Q_1+\widetilde{U}_1^*Q_2\widetilde{U}_1+\ldots +\widetilde{U}_1^*\ldots \widetilde{U}_{n-1}^*Q_n\widetilde{U}_{n-1}\ldots \widetilde{U}_1=I_{\mathcal{D}_{T^*}}$ ,
	\end{enumerate}
	for $1\leq i \leq n$. Let $T_i'=\prod_{j \neq i} T_j$ for $i=1, \dots , n$. As in the proof of Theorem \ref{main}, if we denote $F_i=P_i^{\perp}U_i^*$ and $F_i'=U_iP_i$, then we have
$F_i^*F_i+F_i'F_i'^*=I_{\mathcal{D}_T}=F_iF_i^*+F_i'^*F_i'$ and $F_iF_i'=F_i'F_i=0$. Also, condition-(4) leads to
$
D_TF_i'F_i'^*D_T=D_{T_i}^2$ for $1\leq i \leq n$ and condition-(1) gives
\begin{equation} \label{eqn:0f4}
D_TT_i=F_iD_T+F_i'^*D_TT.
\end{equation}
Lemma \ref{fundexist} guarantees the existence of operators $G_1,\ldots ,G_n,G_1',\ldots ,G_n'\in \mathcal{B}(\mathcal{D}_{T^*})$ satisfying
\begin{equation} \label{eqn:0f5}
D_{T^*}G_iD_{T^*}=D_{T_i'^*}^2T_i^*, \; \& \;\; D_{T^*}G_i'D_{T^*}=D_{T_i^*}^2T_i'^*
\end{equation}
and Lemma \ref{Fundamental equations} further shows that they satisfy
\begin{equation} \label{eqn:0f2} 
	D_{T^*}T_i^*=G_iD_{T^*}+G_i'^*D_{T^*}T^*.
\end{equation}
Following the converse part of the proof of Theorem \ref{main} we have from $(a)$ and $(a')$ the identities $D_{T_i'}^2T_i=T_i-T_i'^*T=D_TF_iD_T$ and $D_{T_i}^2T_i'=T_i'-T_i^*T=D_TF_i'D_T$ respectively. Again, Lemma \ref{fundexist} guarantees the uniqueness of such $F_i$ and $F_i'$. Thus, we have that $F_i=P_i^{\perp}U_i^*$ and $F_i'=U_iP_i$ satisfy
\begin{equation}\label{defFiFi'}
D_TF_iD_T=D_{T_i'}^2T_i \quad \& \quad D_TF_i'D_T = D_{T_i}^2T_i'.
\end{equation}
From Lemma \ref{Fundamental equations} we have that $F_i$, $F_i'$ satisfy
\begin{equation}\label{eqn:0f7}
D_TT_i'=F_i'D_T+F_i^*D_TT.
\end{equation}
Also, from Lemma \ref{lem:main-01} we have that
\begin{equation} \label{eqn:new-01}
D_TF_iF_i^*D_T=D_{T_i'}^2.
\end{equation}
Set $\widetilde{U}_i=G_i^*+G_i'$, $\widetilde{U}_i'=G_i^*-G_i'$ and $Q_i=\dfrac{1}{2}(I-\widetilde{U}_i'^*\widetilde{U}_i)$. We prove that $\widetilde{U}_i$ is a unitary and $Q_i$ is a projection for each $1\leq i \leq n$. Then we have 
\begin{equation} \label{eqn:new-01}
G_iG_i^*+G_i'^*G_i'=I_{\mathcal{D}_{T^*}}=G_i^*G_i+G_i'G_i'^* \quad \& \quad G_iG_i'=G_i'G_i=0.
\end{equation}
See the Appendix for a proof of (\ref{eqn:new-01}). Now
\[
\widetilde{U}_i^*\widetilde{U}_i =(G_i+G_i'^*)(G_i^*+G_i')=G_iG_i^*+G_iG_i'+G_i'^*G_i^*+G_i'^*G_i'=G_iG_i^*+G_i'^*G_i'=I
\] 
and 
\[
	\widetilde{U}_i\widetilde{U}_i^*=(G_i^*+G_i')(G_i+G_i'^*)=G_i^*G_i+G_i^*G_i'^*+G_i'G_i+G_i'G_i'^*=G_i^*G_i+G_i'G_i'^*=I.
\]
Hence $\widetilde{U}_i$ is a unitary for each $i.$
Observe that $G_i=(\widetilde{U}_i^*+\widetilde{U}_i'^*)/2$ and $G_i'=(\widetilde{U}_i-\widetilde{U}_i')/2$. Hence $G_i'G_i=0$ implies that $\widetilde{U}_i\widetilde{U}_i'^*=\widetilde{U}_i'\widetilde{U}_i^*$ and $G_iG_i'=0$ implies that $\widetilde{U}_i^*\widetilde{U}_i'=\widetilde{U}_i'^*\widetilde{U}_i$. Thus $Q_i=\dfrac{1}{2}(I-\widetilde{U}_i^*\widetilde{U}_i')$.
Also, we have that
\[
	Q_i^2=\dfrac{1}{2}(I-\widetilde{U}_i'^*\widetilde{U}_i)\dfrac{1}{2}(I-\widetilde{U}_i^*\widetilde{U}_i')=\dfrac{1}{4}(I-2\widetilde{U}_i'^*\widetilde{U}_i+I)=Q_i \quad \& \quad Q_i^*=\dfrac{1}{2}(I-\widetilde{U}_i^*\widetilde{U}_i')=Q_i.
\]
Therefore, each $Q_i$ is a projection. It follows from here that $G_i'=\widetilde{U}_iQ_i, G_i=Q_i^{\perp}\widetilde{U}_i^*$ for each $i$. Substituting the values of $G_i, G_i'$ in (\ref{eqn:0f2}) and (\ref{eqn:0f3}) (see the Appendix) we obtain our desired identities $(1')$ and $(4')$ respectively. In order to have $(2'), (3')$ it suffices if we prove $[G_i,G_j]=[G_i',G_j']=0$ for all $1\leq i,j \leq n$. We have from Theorem \ref{main} that $(V_1, \dots , V_n)$ is an isometric dilation of $(T_1, \dots , T_n)$. 
Note that
\[
	D_{T^*}(G_i^*T-TF_i)D_{T}=D_{T^*}G_i^*D_{T^*}T-TD_{T}F_iD_T	=(T_i-TT_i'^*)T-T(T_i-T_i'^*T)=0.
\]
Since $TD_T=D_{T^*}T$ i.e. $T$ maps $\mathcal D_T$ into $\mathcal D_{T^*}$, we have that
\begin{equation}\label{relGi-Fi}
	G_i^*T|_{\mathcal{D}_T}=TF_i|_{\mathcal{D}_T}.
\end{equation}
Using $\eqref{defFiFi'}, \eqref{eqn:0f5} \text{ and the fact that } T^*D_{T^*}=D_TT^*$ we have that
 \[
 D_TF_i'T^*D_{T^*}+D_{T^*}G_iD_{T^*}=D_{T_i}^2T_i'T^*+D_{T_i'^*}^2T_i^*=T_i'T^*-T_i^*T_iT_i'T^*+T_i^*-T_i'T_i'^*T_i^* =T_i^*D_{T^*}^2.
 \]
Therefore, for all $i=1,\ldots ,n$ we have
\begin{equation}\label{DT*Gi}
	D_{T^*}G_i=T_i^*D_{T^*}-D_TF_i'T^*.
\end{equation}
Again, by the commutativity of $V_1, \dots , V_n, V_1', \dots, V_n'$, the operators $F_i, F_j, F_i', F_j'$ satisfy
\begin{equation}\label{Ficomrel}
	[F_i,F_j]=[F_i',F_j']=0 \text{ and } [F_i^*,F_j']=[F_j^*,F_i'].
\end{equation}
	Now we prove that $[G_i,G_j]=[G_i',G_j']=0$ and $[G_i,G_j'^*]=[G_j,G_i'^*]$ for all $1\leq i<j \leq n$. First note that $F_i's$ satisfy these commutator relations. The space $\mathcal{K}$ can be decomposed into $\mathcal{K}_1\oplus \mathcal{K}_2$ such that $V|_{\mathcal{K}_1}$ is a pure isometry and $V|_{\mathcal{K}_2}$ is a unitary. We show that $\mathcal{K}_1$, $\mathcal{K}_2$ are reducing subspaces for each $V_i$. If $V_i=\begin{bmatrix}
	A_i&B_i\\C_i&D_i
\end{bmatrix}$ and $V=\begin{bmatrix}
	V_{K1}&0\\
	0&V_{K2}
\end{bmatrix}$ with respect to the decomposition $\mathcal{K}=\mathcal{K}_1\oplus \mathcal{K}_2$, then $V_iV=VV_i$ implies that $B_iV_{K2}=V_{K1}B_i$ and $C_iV_{K1}=V_{K2}C_i$. Therefore, for all $k\in\mathbb{N}$, $V_{K2}^{*k}B_i^*=B_i^*V_{K1}^{*k}$ and $V_{K1}^{*n}C_i^*=C_i^*V_{K2}^{*k}$. Now $V_{K1}$ is a pure isometry and thus for each $h\in \mathcal{H}$, $\|B_i^*V_{K1}^{*k}h\|\to 0$. Again $V_{K2}$ is a unitary and so we have $\|V_{K2}^{*k}B_i^*h\|=\|B_i^*h \|$. Therefore, $V_{K2}^{*k}B_i^*=B_i^*V_{K1}^{*k}$ would imply that $B_i=0 $. Similarly $C_i=0$ for each $i=1,\ldots n$. Hence let $V_i=V_{i1}\oplus V_{i2}$ for all $1\leq i \leq n$. Now $V_{K1}=\Pi_{i=1}^n V_{i1}$. Thus, by Theorem  \ref{BCL} there exist commuting unitaries $\widetilde{U}_{11},\ldots , \widetilde{U}_{n1}$ and projections $Q_{11},\ldots ,Q_{n1}$ in $\mathcal{B}(\mathcal{D}_{V_{K1}^*})$ such that 
$
(V_{11},\ldots ,V_{n1})\equiv (M_{\widetilde{U}_{11}Q_{11}^{\perp}+z\widetilde{U}_{11}Q_{11}}, \ldots , M_{\widetilde{U}_{n1}Q_{n1}^{\perp}+z\widetilde{U}_{n1}Q_{n1}}) 
$
and that
\begin{equation}\label{pure1}
D_{V_{K1}^*}Q_{i1}^{\perp}\widetilde{U}_{i1}^*D_{V_{K1}^*} =D_{V_{i1}'^*}^2V_{i1}^*, 
\end{equation}
\begin{equation}\label{pure2}
D_{V_{K1}^*}\widetilde{U}_{i1}Q_{i1}D_{V_{K1}^*} =D_{V_{i1}^*}^2V_{i1}'^* 
\end{equation}
for each $i=1,\ldots ,n$. The fact that $(M_{\widetilde{U}_{11}Q_{11}^{\perp}+z\widetilde{U}_{11}Q_{11}}, \ldots , M_{\widetilde{U}_{n1}Q_{n1}^{\perp}+z\widetilde{U}_{n1}Q_{n1}})$ is a commuting tuple gives
 \begin{equation}\label{properties} [\widetilde{U}_{i1}Q_{i1}^{\perp},\widetilde{U}_{j1}Q_{j1}^{\perp} ]=0=[\widetilde{U}_{i1}Q_{i1},\widetilde{U}_{j1}Q_{j1}], \, [\widetilde{U}_{i1}Q_{i1}^{\perp},\widetilde{U}_{j1}Q_{j1}]=[\widetilde{U}_{j1}Q_{j1}^{\perp} ,\widetilde{U}_{i1}Q_{i1}],
 \end{equation} 
for $1 \leq i \leq n$. Again, since $\prod_{i=1}^n M_{\widetilde{U}_{i1}Q_{i1}^{\perp}+z\widetilde{U}_{i1}Q_{i1}}=M_z$, we have that 
\begin{equation}\label{properties2}
	Q_{11}+\widetilde{U}_{11}^*Q_{21}\widetilde{U}_{11}+\widetilde{U}_{11}^*\widetilde{U}_{21}^*Q_{31}\widetilde{U}_{21}\widetilde{U}_{11}+\ldots +\widetilde{U}_{11}^*\widetilde{U}_{21}^*\ldots \widetilde{U}_{(n-1)1}^*Q_{n1}\widetilde{U}_{(n-1)1}\ldots \widetilde{U}_{21}\widetilde{U}_{11} =I_{\mathcal{D}_{V_{K1}^*}}
\end{equation} and that $\prod_{i=1}^n\widetilde{U}_{i1}=I_{\mathcal{D}_{V_{{K}_1}^*}}$.
Also, $V_{i2},\ldots , V_{n2}$ are unitary on $\mathcal{K}_2$. It follows that
\[
D_{V_{i}'^*}^2=I_{\mathcal{K}}-\begin{bmatrix}
	V_{i1}'&0\\
	0&V_{i2}'
\end{bmatrix}\begin{bmatrix}
	V_{i1}'^*&0\\
	0&V_{i2}'^*
\end{bmatrix}=\begin{bmatrix}
	I_{\mathcal{K}_1}-V_{i1}'V_{i1}'^*&0\\
	0&I_{\mathcal{K}_2}-V_{i2}'V_{i2}'^*
\end{bmatrix}=\begin{bmatrix}
	I_{\mathcal{K}_1}-V_{i1}'V_{i1}'^*&0\\
	0&0
\end{bmatrix}.
\]
Therefore, $D_{V_{i}'^*}=D_{V_{i1}'^*}\oplus 0$. Similarly we can prove that $D_{V_{i}^*}= D_{V_{i1}^*}\oplus 0$ for all $i=1,\ldots ,n$, with respect to the above decomposition of $\mathcal{K}$. So $D_{V^*}= D_{V_{K1}^*}\oplus0$.  Substituting this we have from \eqref{pure1} and \eqref{pure2}
\[
D_{V_i'^*}^2V_i^*=( D_{V_{i1}'^*}^2\oplus 0)(V_{i1}^*\oplus V_{i2}^*) =D_{V_{i1}'^*}^2V_{i1}^*\oplus 0 = D_{V_{K1}^*}Q_{i1}^{\perp}\widetilde{U}_{i1}^*D_{V_{K1}^*}\oplus 0 =D_{V^*}( Q_{i1}^{\perp}\widetilde{U}_{i1}^*\oplus 0)D_{V^*}.
\]
Hence, for $1\leq i \leq n$ we have  \begin{equation}\label{foruniqueness1}
	D_{V_i'^*}^2V_i^*=D_{V^*}( Q_{i1}^{\perp}\widetilde{U}_{i1}^*\oplus 0)D_{V^*}
\end{equation}
and similarly we have
\begin{equation}\label{foruniqueness2} 
	D_{V_i^*}^2V_i'^*=D_{V^*}( \widetilde{U}_{i1}Q_{i1}\oplus 0)D_{V^*}.
\end{equation}

Note that if $(V,\mathcal{K})$ is the minimal isometric dilation of a contraction $(T,\mathcal{H})$ then the dimensions of $\mathcal{D}_{V^*}$ and $\mathcal{D}_{T^*}$ are equal. Indeed, if $X:\mathcal{D}_{T^*}\to \mathcal{D}_{V^*}$ is defined as $XD_{T^*}h=D_{V^*}h$ for all $h\in \mathcal{H}$ and is extended continuously to the closure, then $X$ is a unitary (see \cite{Nagy}). We briefly recall the proof here. Since $V$ is the minimal isometric dilation of $T$, we have
\[
\mathcal{K}=\overline{span}\{V^kh:k\geq 0,\,h\in \mathcal{H} \}.
\]
Now, for $n\in \mathbb{N}$ and $h\in \mathcal{H}$ we have $D_{V^*}^2 V^nh=(I-VV^*)V^nh=0$. Therefore, we have $D_{V^*}V^nh=0$ for any $n\in \mathbb{N}$ and $h\in \mathcal{H}$. Thus, $\mathcal{D}_{V^*}=\overline{D_{V^*}\mathcal{K}}=\overline{D_{V^*}\mathcal{H}}.$ In fact
\[
\|D_{V^*}h\|^2 =\lbrace(I-VV^*)h,h \rbrace=\|h\|^2-\|V^*h\|^2=\|h\|^2-\|T^*h\|^2=\|D_{T^*}h\|^2.
\]
Therefore, $X$ as defined above is a unitary.
We have that
\begin{equation}\label{claim1}
	D_{V^*}XG_iX^*D_{V^*}=D_{V_i'^*}^2V_i^* \, \text{ and } \, D_{V^*}XG_i'X^*D_{V^*}=D_{V_i^*}^2V_i'^*.
\end{equation}
See the Appendix for a proof of (\ref{claim1}).
Thus, by the uniqueness argument as in Lemma \ref{fundexist}, we have from \eqref{foruniqueness1}, \eqref{foruniqueness2} and \eqref{claim1} that
$ XG_iX^*=  Q_{i1}^{\perp}\widetilde{U}_{i1}^*\oplus 0$ and $ XG_i'X^*=\widetilde{U}_{i1}Q_{i1}\oplus 0$ . This is same as saying that $ XG_iX^*=  Q_{iK}^{\perp}\widetilde{U}_{iK}^*$ and $XG_i'X^*=\widetilde{U}_{iK}Q_{iK} $, where $\widetilde{U}_{iK}=\begin{bmatrix}
	\widetilde{U}_{i1}&0\\0&I_{\mathcal{K}_2}
 \end{bmatrix}$ and $Q_{iK}=Q_{i1}\oplus 0 $ with respect to the decomposition $\mathcal{D}_{V_{K1}^*}\oplus 0 $ of $\mathcal{D}_{V^*}$. Also, from \eqref{properties} it is clear that \[ [\widetilde{U}_{iK}Q_{iK}^{\perp},\widetilde{U}_{jK}Q_{jK}^{\perp} ]=0=[\widetilde{U}_{iK}Q_{iK},\widetilde{U}_{jK}Q_{jK}], \, [\widetilde{U}_{iK}Q_{iK}^{\perp},\widetilde{U}_{jK}Q_{jK}]=[\widetilde{U}_{jK}Q_{jK}^{\perp} ,\widetilde{U}_{iK}Q_{iK}].
\] 
Thus, we have
$
[G_i^*,G_j']=[G_j^*,G_i'],\; \;[G_i,G_i]=0=[G_i',G_j'].
$
So, we have $\widetilde{U}_i\widetilde{U}_j = \widetilde{U}_j \widetilde{U}_i$. From (\ref{properties2}) and the fact that $\prod_{i=1}^n\widetilde{U}_{i1}=I_{\mathcal{D}_{V_{{K}_1}^*}}$, it follow that $\prod_{i=1}^n \widetilde{U}_i=I_{\mathcal D_{T^*}}$. Since $G_i'=\widetilde{U}_iQ_i, G_i=Q_i^{\perp}\widetilde{U}_i^*$, \eqref{properties2} guarantees that condition-$(5)$ holds. The uniqueness of $Q_i$ and $\widetilde{U}_i$ for each $i$ follows from the uniqueness of $G_i$ and $G_i'$. The proof is now complete.
  
\end{proof}

Now we are in a position to present the main unitary dilation theorem which is another main result of this paper.

	\begin{thm} \label{Unimain}
			
	Let $T_1,\ldots,T_n\in \mathcal{B}(\mathcal{H})$ be commuting contractions, $T=\Pi_{j=1}^nT_j$ and $T_i'=\Pi_{i\neq j} T_j$ for $1\leq i \leq n$. 
	\begin{itemize}
		\item[(a)] If $\widetilde{\mathcal{K}}_0$ is a minimal unitary dilation space for $T$, then $(T_1,\ldots ,T_n)$ possesses a unitary dilation $(W_1,\ldots,W_{n})$ on $\widetilde{\mathcal{K}}_0$ with $W=\prod_{i=1}^{n}W_i$ being the minimal unitary dilation of $T$ if and only if there exist unique projections $P_1,\ldots ,P_n$ and unique commuting unitaries $U_1,\ldots ,U_n$ in $\mathcal{B}(\mathcal{D}_T)$ with $\prod_{i=1}^n U_i=I$ such that the following hold for $i=1, \dots, n:$
		\begin{enumerate} 
			\item $D_TT_i=P_i^{\perp}U_i^*D_T+P_iU_i^*D_TT$ ,
			\item  $P_i^{\perp}U_i^*P_j^{\perp}U_j^*=P_j^{\perp}U_j^*P_i^{\perp}U_i^*$ ,
			\item $U_iP_iU_jP_j=U_jP_jU_iP_i$ ,
			\item $D_TU_iP_iU_i^*D_T=D_{T_i}^2$ ,
			\item  $P_1+U_1^*P_2U_1+U_1^*U_2^*P_3U_2U_1+\ldots +U_1^*U_2^*\ldots U_{n-1}^*P_nU_{n-1}\ldots U_2U_1 =I_{\mathcal{D}_{T}}$.
			
		\end{enumerate} 
				
		\item[(b)] Such a unitary dilation is minimal and unique in the sense that if $(X_1, \dots , X_n)$ on $\widetilde{\mathcal K}_1$ is another unitary dilation of $(T_1, \dots , T_n)$ such that the product $X=\prod_{i=1}^nX_i$ is a minimal unitary dilation of $T$, then there is a unitary $\rho:\widetilde{\mathcal K}_0 \rightarrow \widetilde{\mathcal K}_1$ such that $(X_1, \dots , X_n)=(\rho* W_1 \rho, \dots , \rho* W_n \rho)$.
		
	\end{itemize}
	
\end{thm}

\begin{proof}
\textbf{Part-(a). (The }$\Leftarrow$ \textbf{part).}	Suppose there are unique projections $P_1,\ldots ,P_n$ and unique commuting unitaries $U_1,\ldots ,U_n$ in $\mathcal{B}(\mathcal{D}_T)$ with $\prod_{i=1}^nU_i=I$ satisfying the operator identities $(1)-(5)$ for all $1\leq i \leq n$. We explicitly construct a unitary dilation $(W_1, \dots , W_n)$ on $\widetilde{\mathcal K}_0$ of $(T_1, \dots , T_n)$ such that the product $W= \prod_{i=1}^n W_i$ becomes a minimal unitary dilation of the product of the contractions $T=\prod_{i=1}^n$. Evidently, by Theorem \ref{main}, $(T_1, \dots , T_n)$ possesses an isometric dilation on the minimal isometric dilation space $\mathcal K_0$ of $T$, especially when $\mathcal K_0=\HS \oplus l^2(\mathcal D_T)$, the Sch$\ddot{a}$ffer's minimal isometric dilation space of $T$, $(T_1, \dots , T_n)$ dilates to a tuple of commuting isometries $(V_1, \dots , V_n)$ on $\mathcal K_0$, where
	\[
		  V_i=\begin{bmatrix}
		  	T_i&0&0&0&\ldots \\
		  	P_iU_i^*D_T&P_i^{\perp}U_i^*&0&0&\ldots \\
		  	0&P_iU_i^*&P_i^{\perp}U_i^*&0&\ldots \\
		  	0&0&P_iU_i^*&P_i^{\perp}U_i^*&\ldots \\
		  	\ldots & \ldots & \ldots & \ldots & \ldots
		  \end{bmatrix}, \qquad 1\leq i \leq n, 
		  \]
		  and $\prod_{i=1}^nV_i=V$ is the Sch$\ddot{a}$ffer's minimal isometric dilation of $T$. 
	Again, since $(T_1, \dots , T_n)$ possesses an isometric dilation $(V_1, \dots , V_n)$ on $\mathcal K_0$ with  $\prod_{i=1}^nV_i=V$ being the minimal isometric dilation of $T$, it follows from Theorem \ref{mainlemma} that $(T_1^*, \dots , T_n^*)$ possesses an isometric dilation $(Y_1,\dots ,Y_n)$ on the minimal isometric dilation space $\mathcal K_*$ of $T^*$ with $\prod_{i=1}^nY_i=Y$ being the minimal isometric dilation of $T^*$. Therefore, Theorem \ref{main} guarantees the existence of a set of unique orthogonal projections $Q_1,\ldots ,Q_n$ and unique commuting unitaries $\widetilde{U}_1,\ldots ,\widetilde{U}_n$ in $\mathcal{B}(\mathcal{D}_{T^*})$ with $\prod_{i=1}^n\widetilde{U}_i=I$ such that the following identities hold for all $i=1,\dots ,n$:  
	\begin{enumerate} 
		\item[$(1')$] $D_{T^*}T_i^*=Q_i^{\perp}\widetilde{U}_i^*D_{T^*}+Q_i\widetilde{U}_i^*D_{T^*}T^*$,
		\item[$(2')$] $Q_i^{\perp}\widetilde{U}_i^*Q_j^{\perp}\widetilde{U}_j^*=Q_j^{\perp}\widetilde{U}_j^*Q_i^{\perp}\widetilde{U}_i^*$ ,
		\item[$(3')$] $\widetilde{U}_iQ_i\widetilde{U}_jQ_j=\widetilde{U}_jQ_j\widetilde{U}_iQ_i$,
		\item[$(4')$] $D_{T^*}\widetilde{U}_iQ_i\widetilde{U}_i^*D_{T^*}=D_{T_i^*}^2$,
		\item[$(5')$] $Q_1+\widetilde{U}_1^*Q_2\widetilde{U}_1+\ldots +\widetilde{U}_1^*\ldots \widetilde{U}_{n-1}^*Q_n\widetilde{U}_{n-1}\ldots \widetilde{U}_1=I_{\mathcal{D}_{T^*}}$.
		
	\end{enumerate}
	Since $(V_1, \dots , V_n)$ is an isometric dilation of $(T_1,\dots ,T_n)$ on the minimal isometric dilation space of $T$ with $\prod_{i=1}^nV_i=V$ being the minimal isometric dilation of $T$, the $(\Rightarrow)$ part of the Theorem \ref{main} is applicable. From conditions $(a)$ and $(a')$ in the proof of $(\Rightarrow)$ part of the Theorem \ref{main}, we have the identities $D_{T_i'}^2T_i=T_i-T_i'^*T=D_TF_iD_T$ and $D_{T_i}^2T_i'=T_i'-T_i^*T=D_TF_i'D_T$ respectively, where $F_i=P_i^{\perp}U_i^*$ and $F_i'=U_iP_i$. Thus, for $1 \leq i \leq n$ we have
	\begin{equation}\label{DTUIPiDT}
	D_TU_iP_iD_T=T_i'-T_i^*T=D_{T_i}^2T_i' \quad \& \quad D_TP_i^{\perp}U_i^*D_T=T_i-T_i'^*T=D_{T_i'}^2T_i \,.
\end{equation}
An analogous argument holds for $(T_1^*, \dots , T_n^*)$ and thus we have
\begin{equation}\label{DT*UitildeQiDT*}
D_{T^*}\widetilde{U}_iQ_iD_{T^*}=T_i'^*-T_iT^*=D_{T_i^*}^2T_i'^* \quad \& \quad D_{T^*}Q_i^{\perp}\widetilde{U}_iD_{T^*}=T_i^*-T_i'T^*=D_{T_i'^*}^2T_i^*.
\end{equation} 
It is well-known from Sz.-Nagy-Foias theory (see \cite{Nagy}) that any two minimal unitary dilations of a contraction are unitarily equivalent. Thus, without loss of generality we consider the Sch$\ddot{a}$ffer's minimal unitary dilation space $\widetilde{\mathcal K}_0$ of $T$, where
\[
\widetilde{\mathcal{K}}_0=l^2(\mathcal D_T)\oplus \mathcal{H}\oplus l^2(\mathcal{D}_{T^*})=\dots \oplus \mathcal D_{T}\oplus \mathcal D_T \oplus \mathcal D_T \oplus \HS \oplus \mathcal D_{T^*} \oplus \mathcal D_{T^*} \oplus \mathcal D_{T^*} \oplus \dots
\]
and construct a unitary dilation on $\widetilde{\mathcal K}_0$ for $(T_1, \dots , T_n)$. Let us define $W_1,\ldots , W_n$ on $\widetilde{\mathcal K}_0=l^2(\mathcal{D}_T)\oplus\mathcal{H}\oplus l^2(\mathcal{D}_{T^*})$ by 
	\[
	W_i=\left[ 
	\begin{array}{c|c|c} 
		\begin{array}{c c c c} 
			\ddots & \vdots & \vdots & \vdots \\
			\cdots &P_i^{\perp}U_i^*&P_iU_i^*&0\\ 
			\cdots &0 &P_i^{\perp}U_i^*&P_iU_i^*\\
			\cdots &0 &0 &P_i^{\perp}U_i^*
		\end{array}  & \begin{array}{c}
			\vdots \\
			0\\
			0\\
			P_iU_i^*D_T
		\end{array} & \begin{array}{c c c c} 
			\vdots & \vdots & \vdots &  \\
			0 &0 &0&\cdots\\ 
			0 &0 &0&\cdots\\
			-P_iU_i^*T^* &0 &0 &\cdots
		\end{array} \\ 
		\hline 
		\begin{array}{c c c c}
			\cdots\hspace{0.45cm} & 0\hspace{0.8cm} & 0\hspace{0.5cm} & 0\hspace{0.8cm}	
		\end{array} & T_i & \begin{array}{c c c c}
			D_{T^*}\widetilde{U}_iQ_i&0 &0 &\cdots 	
		\end{array}\\
		\hline 
		0 & 0 & \begin{array}{c c c c} 
			\hspace{3mm}	\widetilde{U}_iQ_i^{\perp} &\hspace{3mm} \widetilde{U}_iQ_i &0 & \cdots \\
			0 &\widetilde{U}_iQ_i^{\perp}&\widetilde{U}_iQ_i& \cdots\\ 
			0 &0 &\widetilde{U}_iQ_i^{\perp}& \cdots\\
			\vdots & \vdots & \vdots &\ddots
		\end{array} 
	\end{array} 
	\right] ,\quad 1\leq i \leq n.
	\]
We prove that $(W_1, \dots , W_n)$ is a unitary dilation of $(T_1, \dots , T_n)$. Note that the spaces $\mathcal{K}_0'=l^2(\mathcal{D}_T)\oplus \mathcal{H}$ and $\mathcal{K}_0=\mathcal{H} \oplus l^2(\mathcal{D}_T)$ are isomorphic by the canonical unitary that maps $\xi \oplus h$ to $h \oplus \xi$, where $\xi \in l^2(\mathcal D_T)$ and $h\in \HS$. Hence, this canonically gives a unitary 
\begin{equation}
\label{uni:eq1}
\phi:l^2(\mathcal{D}_T)\oplus\mathcal{H}\oplus l^2(\mathcal{D}_{T^*}) \to  \mathcal K_0 \oplus l^2(\mathcal D_{T^*}).
\end{equation} 
 Now if $\widetilde{W}_i=\phi W_i \phi^*$ on $\mathcal K_0 \oplus l^2(\mathcal D_{T^*})$ is the replica of $W_i$ for each $i$, then it suffices to show that $(\widetilde{W}_1, \dots , \widetilde{W}_n)$ is a unitary dilation of $(T_1, \dots , T_n)$. It is evident that with respect to the decomposition $\mathcal K_0 \oplus l^2(\mathcal D_{T^*})$, 
\begin{equation}
\label{uni:eq2}
\widetilde{W}_i=\phi W_i \phi^*=\begin{bmatrix}
		V_i&D_i\\
		0&E_i
	\end{bmatrix}, \qquad (1\leq i \leq n)
\end{equation}  
where $D_i:l^2(\mathcal{D}_{T^*})\to \mathcal{K}_0$ and $E_i:l^2(\mathcal{D}_{T^*})\to l^2(\mathcal{D}_{T^*}) $ are the following operators:
\[
D_i=\begin{bmatrix}
		D_{T^*}\widetilde{U}_iQ_i&0&0&\cdots\\
		-P_iU_i^*T^*&0&0&\cdots\\
		0&0&0&\cdots\\
		\vdots&\vdots&\vdots&\ddots
	\end{bmatrix} \qquad \& \qquad 
	E_i=\begin{bmatrix}
		\widetilde{U}_iQ_i^{\perp}&\widetilde{U}_iQ_i&0&\cdots\\
		0&\widetilde{U}_iQ_i^{\perp}&\widetilde{U}_iQ_i&\cdots\\
		0&0&\widetilde{U}_iQ_i^{\perp}&\cdots\\
		\vdots&\vdots&\vdots&\ddots
	\end{bmatrix}.
	\]	
	Evidently $\widetilde{W}_i|_{\mathcal K_0}=V_i$ for each $i$ and thus it suffices to show that $(\widetilde{W}_1, \dots , \widetilde{W}_n)$ is a commuting tuple of unitaries, because, then $(\widetilde{W}_1, \dots , \widetilde{W}_n)$ becomes a unitary extension of $(V_1, \dots , V_n)$ and hence a unitary dilation of $(T_1, \dots , T_n)$.
	
	\vsp
	
\noindent\textit{\textbf{Step 1.}}	First we prove that $(\widetilde{W}_1, \dots , \widetilde{W}_n)$ is a commuting tuple. For each $i,j$ we have,
\[ \widetilde{W}_i\widetilde{W}_j=\begin{bmatrix}
	V_iV_j&V_iD_j+D_iE_j\\0&E_iE_j
\end{bmatrix} \quad \text{ and } \quad \widetilde{W}_j\widetilde{W}_i=\begin{bmatrix}
V_jV_i&V_jD_i+D_jE_i\\0&E_jE_i
\end{bmatrix} .
\]
Thus, $\widetilde{W}_i\widetilde{W}_j= \widetilde{W}_j\widetilde{W}_i$ if and only if the following conditions hold:
\begin{enumerate}
\item[(i)] 	$V_iV_j=V_jV_i$,
\item[(ii)]   $V_iD_j+D_iE_j=V_jD_i+D_jE_i$,
\item[(iii)]   $ E_iE_j=E_jE_i$.
\end{enumerate}
 The condition-$(i)$ follows from Theorem \ref{main}. We prove condition-$(iii)$. First we simplify condition-$(2')$ in the statement of this theorem using condition-$(3')$ and the fact that $U_iU_j=U_jU_i$. We have
 $
 \widetilde{U}_i^*Q_j\widetilde{U}_j^*+Q_i\widetilde{U}_i^*\widetilde{U}_j^*=\widetilde{U}_j^*Q_i\widetilde{U}_i^*+Q_j\widetilde{U}_j^*\widetilde{U}_i^*.
$
Since $\widetilde{U}_i,\widetilde{U}_j$ are commuting unitaries, this further gives
\begin{equation}\label{key condition}
	\widetilde{U}_i^*Q_j\widetilde{U}_i+Q_i=\widetilde{U}_j^*Q_i\widetilde{U}_j+Q_j.
\end{equation} 
The condition $\widetilde{U}_i^*Q_j\widetilde{U}_i+Q_i=\widetilde{U}_j^*Q_i\widetilde{U}_j+Q_j\leq I$ then follows from condition-$(5')$. Note that $E_i$ on $l^2(\mathcal{D}_{T^*})$ is equivalent to $T_{\widetilde{U}_iQ_i^{\perp}+\bar{z}\widetilde{U}_iQ_i }$ on $H^2(\mathcal{D}_{T^*})$. Hence by Lemma \ref{BDF lemma alt}, $(iii)$ holds. The proof of $(ii)$, i.e.
\begin{equation} \label{eqn:new-002}
V_iD_j+D_iE_j=V_jD_i+D_jE_i
\end{equation}
is technical and is given in the Appendix. Therefore, $(\widetilde W_1,\ldots ,\widetilde W_n)$ is a commuting tuple. 
  \vsp
  
 \noindent\textit{\textbf{Step 2.}} Now we prove that $\widetilde{W}_i$ is a unitary for each $i=1,2\dots ,n$.
First we note that $\widetilde{W}_i^*\widetilde{W}_i=I$ if and only if
	\[\begin{bmatrix}
		V_i^*V_i&V_i^*D_i\\
		D_i^*V_i&D_i^*D_i+E_i^*E_i
	\end{bmatrix}=\begin{bmatrix}
	I&0\\
	0&I
\end{bmatrix}\] which holds if and only if 
\begin{itemize}
		\item[(a)]$V_i^*V_i=I_{\mathcal{K}_0}$ ,
		\item[(b)]$V_i^*D_i=0 \text{ on } l^2(\mathcal{D}_{T^*})$,
		\item[(c)]$D_i^*D_i+E_i^*E_i=I_{\l^2(\mathcal{D}_{T^*})} $.
	\end{itemize}
	Clearly, $(a)$ follows as each $V_i$ is an isometry. For showing $(b)$ note that 
\[
V_i^*D_i=\begin{bmatrix}
	T_i^*D_{T^*}\widetilde{U}_iQ_i-D_{T}U_iP_iU_i^*T^*&0&0&\ldots\\
	-U_iP_i^{\perp}P_iU_i^*T^*&0&0&\ldots\\
	0&0&0&\ldots\\
	\vdots&\vdots &\vdots &\ddots
\end{bmatrix}.
\]
Clearly, $U_iP_i^{\perp}P_iU_i^*T^*=U_i(I-P_i)P_iU_i^*T^*= 0$ as $P_i$ is a projection. Now
\[
(T_i^*D_{T^*}\widetilde{U}_iQ_i-D_{T}U_iP_iU_i^*T^*)D_{T^*} = T_i^*(T_i'^*-T_iT^*)-D_{T}U_iP_iU_i^*D_TT^*=T_i^*(T_i'^*-T_iT^*)-D_{T_i}^2T^* =0,
\]
where the first and the second equality follow from $\eqref{DT*UitildeQiDT*}$ and condition-(4) of the theorem respectively. This proves $(b)$. For proving $(c)$ we observe that $Q_iQ_i^{\perp}=0$, $\widetilde{U}_i^*\widetilde{U}_i=I$ and $Q_i+Q_i^{\perp}=I$. Further the $(1,1)$ entry of $D_i^*D_i+E_i^*E_i$ with respect to the decomposition $\mathcal{D}_{T^*}\oplus \mathcal{D}_{T^*}\oplus \cdots $ is
$
Q_i\widetilde{U_i}^*D_{T^*}^2\widetilde{U_i}Q_i+TU_iP_iU_i^*T^*+Q_i^{\perp}\widetilde{U_i}^*\widetilde{U_i}Q_i^{\perp}.
$ 
So, we have
\begin{align*}
	Q_i\widetilde{U_i}^*D_{T^*}^2\widetilde{U_i}Q_i+TU_iP_iU_i^*T^*+Q_i^{\perp}\widetilde{U_i}^*\widetilde{U_i}Q_i^{\perp}
	&=Q_i\widetilde{U_i}^*\widetilde{U_i}Q_i-Q_i\widetilde{U_i}^*TT^*\widetilde{U_i}Q_i+TU_iP_iU_i^*T^*+Q_i^{\perp} \\ & =Q_i+Q_i^{\perp}\hspace{1cm}[\text{by }\eqref{relbetUiPiQiUitilde}]\\
	&=I.
\end{align*}
 Thus, $(c)$ holds. We now prove $\widetilde{W}_i\widetilde{W}_i^*=I$. Note that $\widetilde{W}_i\widetilde{W}_i^*=I$ if and only if 
\begin{equation} \label{eqn:new-003}
\begin{bmatrix}
	V_iV_i^*+D_iD_i^*&D_iE_i^*\\
	E_iD_i^*&E_iE_i^*
\end{bmatrix}=\begin{bmatrix}
	I&0\\
	0&I
\end{bmatrix}
\end{equation}
and a proof to this is similar to that of $\widetilde{W}_i^*\widetilde{W}_i=I$ and is given in the Appendix. Thus, $\widetilde{W}_i$ is a unitary for each $i=1, \dots , n$. Hence, $(\widetilde{W}_1,\ldots ,\widetilde{W}_n)$ is a unitary dilation of $(T_1,\ldots ,T_n)$. 

\vsp

\noindent\textit{\textbf{Step  3.}} Now we prove that $\prod_{i=1}^nW_i =W$, where $W$ is the Sch$\ddot{a}$ffer's minimal unitary dilation of $T$, i.e. 
\begin{equation} \label{eqn:uni-dil}
	W =\left[
\begin{array}{ c c c c|c|c c c c}
\bm{\ddots}&\vdots &\vdots&\vdots   &\vdots  &\vdots& \vdots&\vdots&\vdots\\
\cdots&0&I&0  &0&  0&0&0&\cdots\\
\cdots&0&0&I  &0&  0&0&0&\cdots\\
\cdots&0&0&0  &D_T&  -T^*&0&0&\cdots\\ \hline

\cdots&0&0&0   &T&   D_{T^*}&0&0&\cdots\\ \hline

\cdots&0&0&0   &0&  0& I&0&\cdots\\
\cdots&0&0&0   &0&  0&0&I&\cdots\\
\cdots&0&0&0  &0&   0& 0&0&\cdots\\
\vdots&\vdots&\vdots&\vdots&\vdots&\vdots&\vdots&\vdots&\bm{\ddots}\\
\end{array} \right].
	\end{equation}
Let $\widetilde{W}= \phi W \phi^*$ on $\mathcal{K}_0 \oplus l^2(\mathcal{D}_{T^*}) $ where, $\phi$ is as in \eqref{uni:eq1}. Then, evidently
\begin{equation}
\label{uni:eq3}
\widetilde{W}=\begin{bmatrix}
	V&D\\
	0&E
\end{bmatrix}, \qquad (1\leq k \leq n)
\end{equation}  
  where, $V$ is the Sch$\ddot{a}$ffer's minimal isometric dilation of $T$, $D:l^2(\mathcal{D}_{T^*})\to \mathcal{K}_0$ and $E:l^2(\mathcal{D}_{T^*})\to l^2(\mathcal{D}_{T^*}) $ are the following operators:
\begin{equation} \label{uni:eq4}
D=\begin{bmatrix}
	D_{T^*}&0&0&\cdots\\
	-T^*&0&0&\cdots\\
	0&0&0&\cdots\\
	\vdots&\vdots&\vdots&\ddots
\end{bmatrix} \quad \& \quad 
E=\begin{bmatrix}
	0&I&0&\cdots\\
	0&0&I&\cdots\\
	0&0&0&\cdots\\
	\vdots&\vdots&\vdots&\ddots
\end{bmatrix}.
\end{equation}
It follows from \eqref{uni:eq2} that, 
\[ \prod_{i=1}^n \widetilde{W}_i= \prod_{i=1}^n \phi W_i \phi^* = \phi \left(\prod_{i=1}^nW_i \right) \phi^*. \] Therefore, $W=\prod_{i=1}^nW_i$ if and only if $ \widetilde{W}= \phi W \phi^* =\prod_{i=1}^n \widetilde{W}_i$. Hence it suffices to prove that $\widetilde{W}=\prod_{i=1}^n \widetilde{W}_i$.
It was proved by Bercovici, Douglas and Foias in \cite{Berc:Dou:Foi}, all terms that are involved in condition-$(5)$ (and in condition-$(5')$) are mutually orthogonal projections. This is because, the sum of projections is a projection if and only if they are mutually orthogonal. Suppose, for any $1\leq k \leq n$,  
\[
\underline{T_k}=T_1\ldots T_k ,\;\underline{U_k}=U_1\ldots U_k,\; \underline{P_k}=P_1+\underline{U_1}^*P_2\underline{U_1}+\ldots +\underline{U_{k-1}}^*P_k\underline{U_{k-1}},$$ 
$$\underline{\widetilde{U}_k}=\widetilde{U}_1\ldots \widetilde{U}_k,\; \underline{Q_k}=Q_1+\underline{\widetilde{U}_1}^*Q_2\underline{\widetilde{U}_1}+\ldots +\underline{\widetilde{U}_{k-1}}^*Q_k\underline{\widetilde{U}_{k-1}},
\]
where $U_k,P_k\in \mathcal{B}(\mathcal{D}_T)$ are as in the hypothesis of this theorem and $\widetilde{U}_k,Q_k\in \mathcal{B}(\mathcal{D}_{T^*})$ are as obtained in the beginning of the proof satisfying conditions $(1')-(5')$. Evidently each $\underline{U_k}$ is a unitary and each $\underline{P_k}$ is a projection. Further, it follows from the hypothesis of this theorem and from the condition $(5')$ that,
\begin{equation}\label{uni:eq5}
\underline{T}_n=T, \, \underline{U}_n=I_{\mathcal{D}_T},\, \underline{P}_n=I_{\mathcal{D}_T},\, \underline{\widetilde{U}_n}=I_{\mathcal{D}_{T^*}}, \quad \text{ and } \quad \underline{Q}_n=I_{\mathcal{D}_{T^*}}.
\end{equation}
 Define
\[
\underline{V_{k}}=\begin{bmatrix}
	\underline{T_k}&0&0&0&\ldots \\
	\underline{P_k}\;\underline{U_k}^*D_T&\underline{P_k}^{\perp}\underline{U_k}^*&0&0&\ldots \\
	0&\underline{P_k}\;\underline{U_k}^*&\underline{P_k}^{\perp}\underline{U_k}^*&0&\ldots \\
	0&0&\underline{P_k}\;\underline{U_k}^*&\underline{P_k}^{\perp}\underline{U_k}^*&\ldots \\
	\ldots & \ldots & \ldots & \ldots & \ldots
\end{bmatrix}, \qquad 1\leq k \leq n. 
\]    Therefore, we have from \eqref{uni:eq5} that $\underline{V}_n=V$, where $V$ is the Sch$\ddot{a}$ffer's minimal isometric dilation of $T$. Note that for all $k=1, \dots , n$, the operators $V_k$ and $\underline{V_{k}}$ have the following block- matrix form with respect to the decomposition $\mathcal{K}_0=\mathcal{H}\oplus l^2(\mathcal{D}_T)$:
\[ V_k=\begin{bmatrix}
	T_k&0\\
	C_k&S_k
\end{bmatrix}, \ \  \ \ 
\underline{V_{k}}=\begin{bmatrix}
	\underline{T_k}&0\\
	\underline{C_k}&\underline{S_k}
\end{bmatrix}, \] where, \[ C_k=\begin{bmatrix}
	P_kU_k^*D_T\\0\\0\\ \vdots
\end{bmatrix}:\mathcal{H}\to l^2(\mathcal{D}_T),\; S_k=\begin{bmatrix}
	P_k^{\perp}U_k^* &0&0&\cdots\\
	P_kU_k^* & P_k^{\perp}U_k^*&0&\cdots\\
	0&P_kU_k^* & P_k^{\perp}U_k^*&\cdots\\
	\vdots&\vdots&\vdots&\ddots
\end{bmatrix}:l^2(\mathcal{D}_T)\to l^2(\mathcal{D}_T)\]

\[
\underline{C_k}=\begin{bmatrix}
	\underline{P_k}\;\underline{U_k}^*D_T\\0\\0\\ \vdots
\end{bmatrix}:\mathcal{H}\to l^2(\mathcal{D}_T), \ \  \underline{S_k}=\begin{bmatrix}
	\underline{P_k}^{\perp}\underline{U_k}^*&0&0&\cdots\\
	\underline{P_k}\;\underline{U_k}^*&\underline{P_k}^{\perp}\underline{U_k}^*&0&\cdots\\
	0&\underline{P_k}\;\underline{U_k}^*&\underline{P_k}^{\perp}\underline{U_k}^*&\cdots\\
	\vdots&\vdots&\vdots&\ddots
\end{bmatrix}:l^2(\mathcal{D}_T)\to l^2(\mathcal{D}_T).\]
Let  \[
\underline{\widetilde{W}_k}=\begin{bmatrix}
	\underline{V_k}&\underline{D_k}\\
	0&\underline{E_k}
\end{bmatrix}, \qquad (1\leq k \leq n)
\]  
where $\underline{D_k}:l^2(\mathcal{D}_{T^*})\to \mathcal{K}_0$ and $\underline{E_k}:l^2(\mathcal{D}_{T^*})\to l^2(\mathcal{D}_{T^*}) $ are the following operators:
\[
\underline{D_k}=\begin{bmatrix}
	D_{T^*}\underline{\widetilde{U}_k}\;\underline{Q_k}&0&0&\cdots\\
	-\underline{P_k}\;\underline{U_k}^*T^*&0&0&\cdots\\
	0&0&0&\cdots\\
	\vdots&\vdots&\vdots&\ddots
\end{bmatrix} \quad \& \quad 
\underline{E_k}=\begin{bmatrix}
	\underline{\widetilde{U}_k}\;\underline{Q_k}^{\perp}&\underline{\widetilde{U}_k}\;\underline{Q_k}&0&\cdots\\
	0&\underline{\widetilde{U}_k}\;\underline{Q_k}^{\perp}&\underline{\widetilde{U}_k}\;\underline{Q_k}&\cdots\\
	0&0&\underline{\widetilde{U}_k}\;\underline{Q_k}^{\perp}&\cdots\\
	\vdots&\vdots&\vdots&\ddots
\end{bmatrix}.
\]	Again by \eqref{uni:eq3}, \eqref{uni:eq4} and \eqref{uni:eq5}, we have 
\begin{equation}\label{uni:eq6}
\underline{\widetilde{W}_n}=\widetilde{W}.
\end{equation}
We now show that $ \underline{\widetilde{W}_k}\widetilde{W}_{k+1}=\underline{\widetilde{W}_{k+1}}$. Considering the block matrices of these operators with respect to the decomposition $\mathcal{K}_0\oplus l^2(\mathcal{D}_{T^*}) $ it suffices if we prove that
\[
\begin{bmatrix}
	\underline{V_k}V_{k+1}&\underline{V_k}D_{k+1}+\underline{D_k}E_{k+1}\\0&\underline{E_k}E_{k+1}
\end{bmatrix}=\begin{bmatrix}
\underline{V_{k+1}}&\underline{D_{k+1}}\\0&\underline{E_{k+1}}
\end{bmatrix} , \quad 1\leq k \leq n-1.
\]
From Step $3$ in the proof of Theorem \ref{main}, we have that 
$
\underline{V_{k}}V_{k+1}=\underline{V_{k+1}}
$ for $1\leq k \leq n-1$. Again, considering the block matrices of $\underline{V_k}, $ $\underline{V_{k+1}}$ and $V_{k+1}$ with respect to decomposition $\mathcal{K}_0=\mathcal{H}\oplus l^2(\mathcal{D}_T)$ we have $\underline{S_{k+1}}=\underline{S_k}S_{k+1}$. This identity and a simple calculation lead to
 \begin{equation}\label{underlineVk} 
 	\underline{P_k}\;\underline{U_k}^*P_{k+1}U_{k+1}^*=0 \quad \& \quad \underline{P_k}^{\perp}\underline{U_k}^*P_{k+1}U_{k+1}^*+\underline{P_k}\;\underline{U_k}^*P_{k+1}^{\perp}U_{k+1}^*=\underline{P_{k+1}}\;\underline{U_{k+1}}^*.
 \end{equation}  
Again, it is clear from the definition that $\underline{\widetilde{U}_{k+1}}=\underline{\widetilde{U}_k}\widetilde{U}_{k+1}$ and  $\underline{Q_{k+1}}=\underline{Q_k}+\underline{\widetilde{U}_k}^*Q_{k+1}\underline{\widetilde{U}_k}$. Further it can inductively be proved that 
\begin{equation}\underline{Q_k}+\underline{\widetilde{U}_k}^*Q_{k+1}\underline{\widetilde{U}_k}=Q_{k+1}+\widetilde{U}_{k+1}^*\underline{Q_k}\widetilde{U}_{k+1} .
\end{equation}
Hence, it follows from Lemma \ref{BDF lemma 2} that $\underline{E_{k+1}}=\underline{E_k}E_{k+1}$. Thus, 
\begin{equation}\label{underlineEk} \underline{\widetilde{U}_k}\;\underline{Q_k}\widetilde{U}_{k+1}\;Q_{k+1}=0 
\end{equation} and \begin{equation}\label{Ek2}
\underline{\widetilde{U}_{k}}\;\underline{Q_k}^{\perp}\widetilde{U}_{k+1}\;Q_{k+1}+\underline{\widetilde{U}_{k}}\;\underline{Q_k}\;\widetilde{U}_{k+1}\;Q_{k+1}^{\perp}=\underline{\widetilde{U}_{k+1}}\;\underline{Q_{k+1}}.
\end{equation}
It remains to show that
\begin{equation} \label{eqn:new-00011}
\underline{V_k}D_{k+1}+\underline{D_k}E_{k+1}=\underline{D_{k+1}}.
\end{equation}
A proof to this is technical and is given in the Appendix.
 Hence, for each $k$ we have that $\underline{\widetilde{W}_k}\widetilde{W}_{k+1}=\underline{\widetilde{W}_{k+1}}$. Thus, recursively we have $\underline{\widetilde{W}_{n}}=\prod_{i=1}^n\widetilde{W}_{i}.$  It then follows from \eqref{uni:eq6}, that 
$
 \widetilde{W}= \underline{\widetilde{W}_{n}}=\prod_{i=1}^n\widetilde{W}_{i},
 $
 as required.
 
\vsp

\noindent\textbf{(The }$\Rightarrow$ \textbf{part).}	 Suppose $(\widehat{W}_1,\ldots ,\widehat{W}_n)$ is a unitary dilation of $(T_1,\ldots ,T_n)$ on a minimal unitary dilation space ${\mathcal{K}'}$ of $T$ with $\Pi_{i=1}^n\widehat{W}_i=\widehat{W}$ being the minimal unitary dilation of $T$. Then 
\[
{\mathcal{K}'}=\overline{span}\lbrace \widehat{W}^nh: h\in \mathcal{H}, \; n\in \mathbb{Z} \rbrace.
\]
  Let $W$ as in (\ref{eqn:uni-dil}) be the Sch$\ddot{a}$ffer's minimal unitary dilation of $T$ on the Sch$\ddot{a}$ffer's minimal space $\widetilde{\mathcal{K}}_0$. So, 
  \[ \widetilde{\mathcal{K}}_0=\overline{span}\lbrace W^nh: h\in \mathcal{H}, \; n\in \mathbb{Z} \rbrace.\]
  Therefore, the map $\tau : \widetilde{\mathcal{K}}_0\to {\mathcal{K}'}$ defined by $\tau(W^nh)=\widehat{W}^n(h)$ is a unitary which is identity on $\mathcal{H}$. Thus $\mathcal{H}$ is a reducing subspace for $\tau$ and consequently $\tau=\begin{bmatrix} I_{\mathcal{H}}&0\\
	0&\tau_2 \end{bmatrix}$, for some unitary $\tau_2$, with respect to the decomposition $\mathcal{H}\oplus (\widetilde{\mathcal{K}}_0\ominus\mathcal{H})$ of $\widetilde{\mathcal{K}}_0$ and $\mathcal{H}\oplus ({\mathcal{K}'}\ominus\mathcal{H})$ of ${\mathcal{K}'}$. Evidently, $(\tau^* \widehat{W}_1\tau,\ldots ,\tau^*\widehat{W}_n\tau )$ is a commuting tuple of unitaries dilating $(T_1,\ldots ,T_n)$ on $\mathcal{K}'$ with $\prod_{i=1}^n\tau^*\widehat{W}_i\tau=W$ being the Sch$\ddot{a}$ffer's minimal unitary dilation of $T$ as in (\ref{eqn:uni-dil}). It is clear from (\ref{eqn:uni-dil}) that $W$ has the following block-matrix form with respect to the decomposition $\widetilde{\mathcal{K}}_0 = l^2(\mathcal{D}_T)\oplus \mathcal{H}\oplus l^2(\mathcal{D}_{T^*})$:
\[
\begin{bmatrix}
		*&*&*\\
		0&T&*\\
		0&0&*
	\end{bmatrix}.
	\] 
	Suppose $W_i=\tau^*\widehat{W}_i \tau$ for $i=1,\dots, n$. Then Lemma \ref{forunidil} tells us that with respect to the decomposition $l^2(\mathcal{D}_T)\oplus \mathcal{H}\oplus l^2(\mathcal{D}_{T^*})$ of $\widetilde{\mathcal{K}}_0$, each $W_j$ has a matrix representation of the form
	\[\begin{bmatrix}
		*&*&*\\
		0&T_j&*\\
		0&0&*
	\end{bmatrix}. \] 
	It is obvious from the blolck-matrix form that ${\mathcal{H}\oplus l^2(\mathcal{D}_T)}$ is an invariant subspace for each $W_j$ and for $W$. Set $V_j:=W_j|_{\mathcal{H}\oplus l^2(\mathcal{D}_T)}=\begin{bmatrix}
		T_j&0\\
		*&*
	\end{bmatrix}$ and $V=W|_{\mathcal{H}\oplus l^2(\mathcal{D}_T)}$. Then, $V$ is the Sch$\ddot{a}$ffer's minimal isometric dilation of $T$ on the Sch$\ddot{a}$ffer's minimal isometric dilation space ${\mathcal{H}\oplus l^2(\mathcal{D}_T)}$ and $(V_1, \dots , V_n)$ is a commuting isometric tuple that dilates $(T_1, \dots , T_n)$. Since $\prod_{i=1}^n W_i=W$, a simple computation shows that $\prod_{i=1}^n V_i=V$. Thus, $(V_1, \dots , V_n)$ is an isometric dilation of $(T_1, \dots , T_n)$ with $\prod_{i=1}^n V_i=V$ being the minimal isometric dilation of $T$. Therefore, it follows from Theorem \ref{main} that there are unique orthogonal projections $P_1, \dots , P_n$ and unique commuting unitaries $U_1, \dots , U_n$ in $\mathcal B(\mathcal D_P)$ with $\prod_{i=1}^n U_i=I_{\mathcal D_P}$ satisfying the conditions $(1)-(5)$ of this theorem. 
	
	\vsp

	\noindent\textit{\textbf{Part-(b).}}  Since $W= \prod_{i=1}^n W_i$ on $\widetilde{\mathcal K}_0$ is a minimal unitary dilation of $T= \prod_{i=1}^n T_i$, it follows from the definition of minimality that the dilation $(W_1, \dots , W_n)$ of $(T_1, \dots , T_n)$ is minimal. Thus,
\[
\widetilde{\mathcal K}_0=\overline{Span} \, \{W^nh \,: \, h\in \HS \, \, \,  \& \, \, \, n \in \mathbb Z  \} =\overline{Span} \, \{W_1^{d_1}\dots W_{n}^{d_n}h \,: \, h\in \HS \, \, \,  \& \, \, \, d_1, \dots , d_n \in \mathbb Z  \}.
\]	
If $(X_1, \dots , X_n)$ on $\widetilde{\mathcal K}_1$ is a unitary dilation of $(T_1, \dots , T_n)$ with the product $X=\prod_{i=1}^n$ being a minimal unitary dilation of $T$, then this is a minimal unitary dilation of $(T_1, \dots , T_n)$ and consequently we have
\[
\widetilde{\mathcal K}_1 =\overline{Span} \, \{X^nh \,: \, h\in \HS \, \, \,  \& \, \, \, n \in \mathbb Z  \}=\overline{Span} \, \{X_1^{d_1}\dots X_{n}^{d_n}h \,: \, h\in \HS \, \, \,  \& \, \, \, d_1, \dots , d_n \in \mathbb Z  \}.
\]	
Evidently, the unitary $\rho\,: \widetilde{\mathcal K}_0 \rightarrow \widetilde{\mathcal K}_1$ that maps $W^nh$ to $X^nh$ also maps $W_1^{d_1}\dots W_{n}^{d_n}h$ to $X_1^{d_1}\dots X_{n}^{d_n}h$ and this gives $X_i = \rho^* W_i \rho$ for $1\leq i \leq n$ so that the product $X=\prod_{i=1}^n X_i = \rho^* W \rho$. The proof is now complete.
	   
\end{proof}

Theorem \ref{main} shows that a tuple of commuting contraction $(T_1,\dots , T_n)$ admits an isometric dilation $(V_1, \dots , V_n)$ on the minimal isometric dilation space of $T=\prod_{i=1}^n T_i$ if we assume the exact five conditions as in Theorem \ref{Unimain}. Now being a subnormal tuple, $(V_1, \dots ,V_n)$ always extends to a commuting unitary tuple which must be a unitary dilation of $(T_1, \dots , T_n)$. The main achievement here is that we found such a unitary extension on the minimal unitary dilation space of $T$ without assuming any additional conditions. Needless to mention that Theorem \ref{mainlemma} plays the central role in determining this. 

\begin{rem}
	
The minimal unitary dilation of Theorem \ref{Unimain} is not unconditional even for $n=2$, though Ando's theorem tells us that every pair of commuting contractions $(T_1,T_2)$ dilates to a pair of commuting unitaries without any conditions on $(T_1,T_2)$. As we have seen in Example 3.6 in \cite{Sou:Pra} that if we choose
\[
	T_1=\begin{pmatrix}
		0&0&0\\
		1/3&0&0\\
		0&1/3\sqrt{3}&0
	\end{pmatrix} \quad \& \quad
	T_2=\begin{pmatrix}
		0&0&0\\
		0&0&0\\
		-1/\sqrt{3}&0&0
	\end{pmatrix},
	\]
then condition-$(4)$ of Theorem \ref{Unimain} is not satisfied and consequently $(T_1, T_2)$ does not dilate to a pair of commuting unitaries $(W_1,W_2)$ on the minimal unitary dilation space of $T_1T_2$ with $W_1W_2$ being the minimal unitary dilation of $T_1T_2$.

\end{rem}

We conclude this Section with the following analogue of Theorem \ref{mainlemma} for unitary dilation.

\begin{thm} \label{thm:analogue-1}
Let $T_1,\dots T_n\in \mathcal{B}(\mathcal{H})$ be commuting contractions and let $\widetilde{\mathcal K}$ be the minimal unitary dilation space for $T=\Pi_{i=1}^nT_i$. Then $(T_1, \dots , T_n)$ possesses a unitary dilation $(W_1,\dots ,W_n)$ on $\widetilde{\mathcal K}$ with $W=\prod_{i=1}^{n}W_i$ being the minimal unitary dilation of $T$ if and only if $(T_1^*, \dots , T_n^*)$ possesses a unitary dilation $(Z_1,\dots ,Z_n)$ on $\widetilde{\mathcal K_*}$ with $Z=\prod_{i=1}^{n}Z_i$ being the minimal unitary dilation of $T^*$, where $\widetilde{\mathcal K_*}$ is the minimal unitary dilation space for $T^*$.

\end{thm}

\begin{proof}

Suppose $(T_1, \dots , T_n)$ possesses a unitary dilation $(W_1, \dots , W_n)$ on $\widetilde{\mathcal K}$ with $W$ being the minimal unitary dilation of $T$. Then $(T_1, \dots , T_n)$ satisfies conditions $(1)-(5)$ of Theorem \ref{Unimain} and hence by Theorem \ref{main}, it admits an isometric dilation $(V_1, \dots , V_n)$ on the minimal isometric dilation space $\mathcal K$ of $T$ with $V=\prod_{i=1}^n V_i$ being the minimal isometric dilation of $T$. So, by Theorem \ref{mainlemma}, $(T_1^*, \dots , T_n^*)$ possesses an isometric dilation $(Y_1, \dots , Y_n)$ on the minimal isometric dilation space $\mathcal K_*$ of $T^*$ with $Y=\prod_{i=1}^n Y_i$ being the minimal isometric dilation of $T^*$. Again applying Theorem \ref{main} on $(T_1^*, \dots , T_n^*)$ we have that there are unique orthogonal projections $Q_1, \dots , Q_n$ and unique commuting unitaries $\widetilde{U}_1, \dots , \widetilde{U}_n$ in $\mathcal B(\mathcal D_{T^*})$ with $\prod_{i=1}^n \widetilde{U_i}=I_{\mathcal D_{T^*}}$ satisfying the hypotheses of Theorem \ref{main}. Thus, with the same hypotheses we apply Theorem \ref{Unimain} to obtain a unitary dilation $(Z_1, \dots, Z_n)$ of $(T_1^*, \dots , T_n^*)$ on the minimal unitary dilation space $\widetilde{\mathcal K_*}$ of $T^*$.

\end{proof}

\section{Minimal unitary dilation when the product is a $C._0$ contraction}
\vspace{0.4cm}

\noindent In this Section, we show that only four out of five conditions of Theorem \ref{Unimain} are necessary and sufficient for the existence of a unitary dilation of a tuple of commuting contraction $(T_1, \dots , T_n)$ when the product $T=\prod_{i=1}^n T_i$ is a $C._0$ contraction. First we recall a pair of isometric dilation theorems in this setting from \cite{Sou:Pra}. 

\begin{thm}\label{puredil}
	Let $T_1,\ldots ,T_n$ be commuting contractions on a Hilbert space $\mathcal{H}$ such that their product $T=\Pi_{i=1}^nT_i$ is a $C._0$ contraction. Then $(T_1,\ldots ,T_n)$ possesses an isometric dilation $(V_1,\ldots ,V_n)$ on $\mathcal{K}$ with $V=\Pi_{i=1}^nV_i$ being a minimal isometric dilation of $T$ if and only if there are unique orthogonal projections $P_1,\ldots ,P_n$ and unique commuting unitaries $U_1,\ldots ,U_n$ in $\mathcal{B}(\mathcal{D}_{T^*})$ with $\prod_{i=1}^n U_i=I_{\mathcal D_{T^*}}$ such that the following hold for $i=1, \dots , n$ :
	\begin{enumerate}
		\item $D_{T^*}T_i^*=P_i^{\perp}U_i^*D_{T^*}+P_iU_i^*D_{T^*}T^*$,
		\item $P_i^{\perp}U_i^*P_j^{\perp}U_j^*=P_j^{\perp}U_j^*P_i^{\perp}U_i^*$ ,
		\item $U_iP_iU_jP_j=U_jP_jU_iP_i$ ,
		\item  $P_1+U_1^*P_2U_1+U_1^*U_2^*P_3U_2U_1+\ldots +U_1^*U_2^*\ldots U_{n-1}^*P_nU_{n-1}\ldots U_2U_1 =I_{\mathcal{D}_{T^*}}$. 
	\end{enumerate} 
\end{thm}

We have also seen in \cite{Sou:Pra} that an isometric dilation can be constructed with conditions $(1)-(3)$ only of Theorem \ref{puredil}. We present the result below.

\begin{thm}\label{coropuredil}
	Let $T_1,\ldots, T_n \in \mathcal{B}(\mathcal{H})$ be commuting contractions such that $T=\prod_{i=1}^nT_i$ is a $C._0$ contraction. Let $T_i'=\prod_{i\neq j} T_j$ for $1\leq i \leq n$. Then $(T_1,\ldots ,T_n)$ possesses an isometric dilation on the minimal isometric dilation space of $T$, if there are projections $P_1,\ldots ,P_n$ and commuting unitaries $U_1,\ldots ,U_n$ in $\mathcal{B}(\mathcal{D}_{T^*})$ such that the following hold for $i=1, \dots, n$:
	\begin{enumerate} 
		\item $D_TT_i=P_i^{\perp}U_i^*D_T+P_iU_i^*D_TT$ ,
		\item  $P_i^{\perp}U_i^*P_j^{\perp}U_j^*=P_j^{\perp}U_j^*P_i^{\perp}U_i^*$ ,
		\item $U_iP_iU_jP_j=U_jP_jU_iP_i$ .
	\end{enumerate} 
	Conversely, if a commuting tuple of contractions $(T_1,\ldots ,T_n)$, with the product $T=\prod_{i=1}^n T_i$ being pure, possesses an isometric dilation $(\widehat{V}_1,\ldots,\widehat{V}_{n})$, where $\hat{V}=\prod_{i=1}^n\hat{V}_i$ is the minimal isometric dilation of $T$, then there are unique projections $P_1,\ldots ,P_n$ and unique commuting unitaries $U_1,\ldots ,U_n$ in $\mathcal{B}(\mathcal{D}_{T^*})$ satisfying the conditions $(1)-(3)$ above. 
\end{thm}

We now present the unitary dilation theorem and this is the main result of this Section.

\begin{thm}\label{purediluni}
	Let $T_1,\ldots ,T_n$ be commuting contractions on a Hilbert space $\mathcal{H}$ such that their product $T=\Pi_{i=1}^nT_i$ is a $C._0$ contraction. Then $(T_1,\ldots ,T_n)$ possesses a unitary dilation $(W_1,\ldots ,W_n)$ on $\widetilde{\mathcal K}$ with $W=\prod_{i=1}^nW_i$ being the minimal unitary dilation of $T$ on $\widetilde{\mathcal K}$ if and only if there are unique projections $Q_1,\ldots ,Q_n$ and unique commuting unitaries $\widetilde{U}_1,\ldots ,\widetilde{U}_n$ in $\mathcal{B}(\mathcal{D}_{T^*})$ such that $\prod_{i=1}^n\widetilde{U}_i=I_{\mathcal{D}_{T^*}}$ and the following conditions hold for $i=1,\dots , n$:
	\begin{enumerate}
		\item $D_{T^*}T_i^*=Q_i^{\perp}\widetilde{U}_i^*D_{T^*}+Q_i\widetilde{U}_i^*D_{T^*}T^*$,
		\item $Q_i^{\perp}\widetilde{U}_i^*Q_j^{\perp}\widetilde{U}_j^*=Q_j^{\perp}\widetilde{U}_j^*Q_i^{\perp}\widetilde{U}_i^*$ ,
		\item $\widetilde{U}_iQ_i\widetilde{U}_jQ_j=\widetilde{U}_jQ_j\widetilde{U}_iQ_i$,
	\item  $Q_1+\widetilde{U}_1^*Q_2\widetilde{U}_1+\widetilde{U}_1^*\widetilde{U}_2^*Q_3\widetilde{U}_2\widetilde{U}_1+\ldots +\widetilde{U}_1^*\widetilde{U}_2^*\ldots \widetilde{U}_{n-1}^*Q_n\widetilde{U}_{n-1}\ldots \widetilde{U}_2\widetilde{U}_1 =I_{\mathcal{D}_T}$. 
	\end{enumerate} 
\end{thm}
\begin{proof}
	First suppose that there are projections $Q_1,\ldots ,Q_n$ and commuting unitaries $\widetilde{U}_1,\ldots ,\widetilde{U}_n$ in $\mathcal{B}(\mathcal{D}_{T^*})$ satisfying the given conditions. Then following the proof of Theorem \ref{puredil} from \cite{Sou:Pra} we have that the Toeplitz operator tuple  $(T_{\widetilde{U}_1Q_1^{\perp}+z\widetilde{U}_1Q_1},\ldots ,T_{\widetilde{U}_nQ_n^{\perp}+z\widetilde{U}_nQ_n})$ on $H^2(\mathcal{D}_{T^*})$ is an isometric dilation of $(T_1, \dots , T_n)$ with their product $\prod_{i=1}^n T_{\widetilde{U}_1Q_1^{\perp}+z\widetilde{U}_1Q_1}=T_z$ being the minimal isometric dilation of the $C._0$ contraction $T$. It is obvious that $(W_1, \dots , W_n)=(M_{\widetilde{U}_1Q_1^{\perp}+z\widetilde{U}_1Q_1},\ldots ,M_{\widetilde{U}_nQ_n^{\perp}+z\widetilde{U}_nQ_n})$ on $L^2(\mathcal{D}_{T^*})$ is a unitary extension of $(M_{\widetilde{U}_1Q_1^{\perp}+z\widetilde{U}_1Q_1},\ldots ,M_{\widetilde{U}_nQ_n^{\perp}+z\widetilde{U}_nQ_n})$ on $H^2(\mathcal{D}_{T^*})$. So, $(W_1, \dots , W_n)$ on $L^2(\mathcal D_{T^*})$ is a unitary dilation of $(T_1, \dots , T_n)$. Needless to mention that $\prod_{i=1}^n W_i=M_z$ is the minimal unitary dilation of $T$.

Conversely, suppose $ (W_1,\ldots,W_n)$ is a unitary dilation of $(T_1,\ldots ,T_n)$ on the minimal unitary dilation space $\widetilde{\mathcal{K}}$ of $T$ with $\prod_{i=1}^nW_i=W$ being the minimal unitary dilation of $T$. Without loss of generality let $\widetilde{\mathcal{K}}=l^2(\mathcal{D}_T)\oplus \mathcal{H}\oplus l^2(\mathcal{D}_{T^*})$. Then by converse part of Theorem \ref{Unimain}, there are unique projections $P_1,\dots ,P_n$ and unique commuting unitaries $U_1,\dots ,U_n$ in $\mathcal{B}(\mathcal{D}_{T})$ such that $\prod_{i=1}^nU_i=I_{\mathcal{D}_{T}}$ and conditions $(1)-(5)$ in the statement of Theorem \ref{Unimain} are satisfied. Thus $(\Leftarrow)$ part of Theorem \ref{main} tells us that $(T_1,\dots ,T_n)$ possesses an isometric dilation $(V_1,\dots ,V_n)$ on the minimal isometric dilation space of $T$ with $\prod_{i=1}^nV_i$ being the minimal isometric dilation of $T$. Then by Theorem \ref{mainlemma}, $(T_1^*,\dots ,T_n^*)$ possesses an isometric dilation $(Y_1,\dots ,Y_n)$ on the minimal isometric dilation space of $T^*$ with $\prod_{i=1}^nY_i$ being the minimal isometric dilation of $T^*$. Hence, by an application of Theorem \ref{main} again we have that there are unique projections $Q_1,\ldots ,Q_n$ and unique commuting unitaries $\widetilde{U}_1,\ldots ,\widetilde{U}_n$ in $\mathcal{B}(\mathcal{D}_{T^*})$ such that $\prod_{i=1}^n\widetilde{U}_i=I_{\mathcal{D}_{T^*}}$ such that the conditions $(1)-(4)$ are satisfied. Hence the proof is complete.
 
\end{proof}

We now present an analogue of Theorem \ref{coropuredil} for a unitary dilation when the product $T$ is a $C._{0}$ contraction.
 
\begin{thm}\label{coropurediluni}
	Let $T_1,\ldots, T_n \in \mathcal{B}(\mathcal{H})$ be commuting contractions such that $T=\prod_{i=1}^nT_i$ is a $C._0$ contraction. Let $T_i'=\prod_{i\neq j} T_j$ for $1\leq i \leq n$. Then $(T_1,\ldots ,T_n)$ possesses a unitary dilation on the minimal unitary dilation space of $T$, if there are projections $P_1,\ldots ,P_n$ and commuting unitaries $U_1,\ldots ,U_n$ in $\mathcal{B}(\mathcal{D}_{T^*})$ such that the following hold for $i=1, \dots, n$:
	\begin{enumerate} 
		\item $D_TT_i=P_i^{\perp}U_i^*D_T+P_iU_i^*D_TT$ ,
		\item  $P_i^{\perp}U_i^*P_j^{\perp}U_j^*=P_j^{\perp}U_j^*P_i^{\perp}U_i^*$ ,
		\item $U_iP_iU_jP_j=U_jP_jU_iP_i$.
	\end{enumerate} 
	Conversely, if a commuting tuple of contractions $(T_1,\ldots ,T_n)$, with the product $T=\prod_{i=1}^n T_i$ being a $C._0$ contraction, possesses a unitary dilation $(W_1,\ldots,W_{n})$, where $W=\prod_{i=1}^nW_i$ is the minimal unitary dilation of $T$, then there are unique projections $P_1,\ldots ,P_n$ and unique commuting unitaries $U_1,\ldots ,U_n$ in $\mathcal{B}(\mathcal{D}_{T^*})$ satisfying the conditions $(1)-(3)$ above. 
\end{thm}
\begin{proof}
	The existence of such a set of projections and unitaries guarantees the existence of an isometric dilation $(V_1,\dots ,V_n)$ on $H^2(\mathcal{D}_{T^*})$ of $(T_1, \dots , T_n)$ by Theorem \ref{coropuredil}. Following the proof of Theorem \ref{puredil} from in \cite{Sou:Pra}, we see that $V_i= T_{U_iP_i^{\perp}+zU_iP_i}$ for $1\leq i \leq n$. Thus, if we set $W_i:= M_{U_iP_i^{\perp}+zU_iP_i}$ on $L^2(\mathcal{D}_{T^*})$, then $(W_1,\dots ,W_n)$ is a unitary extension of $(V_1,\dots ,V_n)$ and hence is a unitary dilation of $(T_1,\dots, T_n)$. The converse part follows from Theorem \ref{purediluni}.  
	 
\end{proof}


\section{Sz. Nagy-Foias type minimal unitary dilation}

\vspace{0.4cm}

\noindent Suppose $T\in \mathcal{B}(\mathcal{H})$ is a c.n.u. contraction and $V$ on $\mathcal{K}_{0}$ is the minimal isometric dilation of $T$. By Wold decomposition there are reducing subspaces $\mathcal{K}_{01}$, $\mathcal{K}_{02}$ of $V$ such that $\mathcal{K}_0=\mathcal{K}_{01}\oplus \mathcal{K}_{02}$, $V|_{\mathcal{K}_{01}}$ is the unilateral shift and $V|_{\mathcal{K}_{02}}$ is a unitary. Then $\mathcal{K}_{01}$ can be identified with $H^2(\mathcal D_{T^*})$ and $\mathcal{K}_{02}$ can be identified with $\ov{\Delta_{T}(L^2(\mathcal{D}_T))}$, where $\Delta_{T}(t)=[I_{\mathcal D_T}-\Theta_T(e^{it})^*\Theta_T(e^{it})]^{1/2}$, where $\Theta_T$ is the characteristic function of the contraction $T$. For further details see Chapter-VI of \cite{Nagy}. Thus, $\mathcal{K}_{0}= \mathcal{K}_{01}\oplus \mathcal{K}_{02}$ can be identified with $\mathbb{K}_{+}=H^2(\mathcal D_{T^*})\oplus \overline{\Delta_{T}(L^2(\mathcal D_T))}$. Also, $V$ on $\mathcal{K}_0$ can be realized as $M_z \oplus M_{e^{it}}|_{\overline{\Delta_T(L^2(\mathcal D_T))}}$. Thus, there is a unitary
\begin{equation}\label{unitary between to spaces}
	\tau=\tau_1\oplus \tau_2:\mathcal{K}_{01}\oplus \mathcal{K}_{02}\to (H^2\otimes \mathcal{D}_{T^*})\oplus \overline{\Delta_{T}(L^2(\mathcal{D}_T))}\, :=\widetilde{\mathbb{K}}_{+}
\end{equation}
such that $V$ on $\mathcal{K}_0$ can be realized as $(M_z \otimes I_{\mathcal D_{T^*}}) \oplus M_{e^{it}}|_{\overline{\Delta_T(L^2(\mathcal D_T))}}$ on $\widetilde{\mathbb{K}}_{+}$.

If $(T_1,\ldots ,T_n)$ is a commuting tuple of contractions on $\mathcal{H}$ satisfying the hypotheses of Theorem \ref{coro-main}, then it dilates to a commtuing tuple of isometries $(V_1,\ldots ,V_n)$ on minimal isometric dilation space $\mathcal{K}_0$ of $T$. Now by Wold decomposition of commuting isometries, we have that $\mathcal{K}_{01}$ and $\mathcal{K}_{02}$ are reducing subspaces for each $V_i$ and that \begin{equation}\label{infoV_i2}
V_{i2}=V_i|_{\mathcal{K}_{02}}
\end{equation}
is a unitary for $1\leq i \leq n$. We have from \cite{Sou:Pra} the following useful analogue of Theorem \ref{coro-main} where we consider the Sz. Nagy-Foias minimal isometric dilation space of $T$.
 
\begin{thm}[\cite{Sou:Pra}, Theorem 6.1] \label{Nagy isodil}
	Let $(T_1,\ldots ,T_n)$ be a tuple of commuting contractions acting on $\mathcal{H}$ such that $T=\Pi_{i=1}^nT_i$ is a c.n.u. contraction. Suppose there are projections $P_1,\ldots  ,P_n$ and commuting unitaries $U_1,\ldots ,U_n $ in $\mathcal{B}(\mathcal{D}_T)$ satisfying
	\begin{enumerate}
		\item $D_TT_i=P_i^{\perp}U_i^*D_T+P_iU_i^*D_TT$
		\item  $P_i^{\perp}U_i^*P_j^{\perp}U_j^*=P_j^{\perp}U_j^*P_i^{\perp}U_i^*$,
		\item $U_iP_iU_jP_j=U_jP_jU_iP_i$,
		\item $D_TU_iP_iU_i^*D_T=D_{T_i}^2$,
	\end{enumerate} for $1\leq i<j\leq n$. 
	Then there are projections $Q_1,\ldots ,Q_n$ and commuting unitaries $\widetilde{U}_1,\ldots,\widetilde{U}_n$ in $\mathcal{B}(\mathcal{D}_{T^*})$ such that $(T_1, \dots , T_n)$ dilates to the tuple of commuting isometries $(\widetilde{V}_{11}\oplus \widetilde{V}_{12},\ldots ,\widetilde{V}_{n1}\oplus \widetilde{V}_{n2})$ on $\widetilde{\mathbb{K}}_+=H^2\otimes \mathcal{D}_{T^*}\oplus \ov{\Delta_{T}(L^2(\mathcal{D}_T))}$, where\begin{align*}
		\widetilde{V}_{i1}&=I\otimes\widetilde{U}_iQ_i^{\perp}+M_z\otimes \widetilde{U}_iQ_i\,,\\
		\widetilde{V}_{i2}&=\tau_2V_{i2}\tau_2^* \,,
	\end{align*} for unitaries $\tau_2$ and $V_{i2}$ as in \eqref{unitary between to spaces} and \eqref{infoV_i2} respectively for $1 \leq i \leq n$.
	 
\end{thm}

Now we present an analogue of this isometric dilation theorem in the unitary dilation setting. This is a main result of this Section.

\begin{thm} \label{thm:Nagy-Foias}
	Let $(T_1,\ldots ,T_n)$ be a tuple of commuting contractions acting on $\mathcal{H}$ such that $T=\Pi_{i=1}^nT_i$ is a c.n.u. contraction. Suppose there are projections $P_1,\ldots  ,P_n$ and commuting unitaries $U_1,\ldots ,U_n $ in $\mathcal{B}(\mathcal{D}_T)$ satisfying
	\begin{enumerate}
		\item $D_TT_i=P_i^{\perp}U_i^*D_T+P_iU_i^*D_TT$
		\item  $P_i^{\perp}U_i^*P_j^{\perp}U_j^*=P_j^{\perp}U_j^*P_i^{\perp}U_i^*$,
		\item $U_iP_iU_jP_j=U_jP_jU_iP_i$,
		\item $D_TU_iP_iU_i^*D_T=D_{T_i}^2$,
	\end{enumerate} for $1\leq i<j\leq n$. 
	Then there are projections $Q_1,\ldots ,Q_n$ and commuting unitaries $\widetilde{U}_1,\ldots,\widetilde{U}_n$ in $\mathcal{B}(\mathcal{D}_{T^*})$ such that $(T_1, \dots , T_n)$ dilates to a tuple of commuting unitaries $(\widetilde{W}_{11}\oplus \widetilde{W}_{12},\ldots ,\widetilde{W}_{n1}\oplus \widetilde{W}_{n2})$ on $\widetilde{\mathbb{K}}=L^2\otimes \mathcal{D}_{T^*}\oplus \ov{ \Delta_{T}(L^2(\mathcal{D}_T))}$, where
	\begin{align*}
		\widetilde{W}_{i1}&=I\otimes\widetilde{U}_iQ_i^{\perp}+M_z\otimes \widetilde{U}_iQ_i\,,\\
		\widetilde{W}_{i2}&=\tau_2V_{i2}\tau_2^* \,, 
	\end{align*} for unitaries $\tau_2$ and $V_{i2}$ as in \eqref{unitary between to spaces} and \eqref{infoV_i2} respectively for $1 \leq i \leq n$. 
\end{thm}

\begin{proof}
	Since there are projections $P_1,\ldots  ,P_n\in \mathcal{B}(\mathcal{D}_T)$ and commuting unitaries $U_1,\ldots ,U_n\in \mathcal{B}(\mathcal{D}_T)$ satisfying conditions $(1)-(4)$, we have by Theorem \ref{Nagy isodil} that $(T_1, \dots , T_n)$ has an isometric dilation $(\widetilde{V}_{11}\oplus \widetilde{V}_{12},\ldots ,\widetilde{V}_{n1}\oplus \widetilde{V}_{n2})$ on $\widetilde{\mathbb{K}}_{+}=H^2\otimes \mathcal{D}_{T^*}\oplus \ov{\Delta_{T}(L^2(\mathcal{D}_T))}$, where
	\begin{align*}
		\widetilde{V}_{i1}&=I\otimes\widetilde{U}_iQ_i^{\perp}+M_z\otimes \widetilde{U}_iQ_i\,, \\
		\widetilde{V}_{i2}&=\tau_2V_{i2}\tau_2^*, \quad (1\leq i \leq n) 
	\end{align*}
	for a unitary $\tau_2:\mathcal{K}_{02}\to \ov{\Delta_{T}(L^2(\mathcal{D}_T))}$ and a unitary $V_{i2}$ on $\mathcal{K}_{02}$ as in \eqref{unitary between to spaces} and \eqref{infoV_i2} respectively.
Let us consider for $i=1, \dots , n$ the following operators on $\widetilde{\mathbb{K}}=L^2\otimes \mathcal{D}_{T^*}\oplus \ov{\Delta_{T}(L^2(\mathcal{D}_T))}$:
\begin{align*}
	\widetilde{W}_{i1}&=I\otimes\widetilde{U}_iQ_i^{\perp}+M_z\otimes \widetilde{U}_iQ_i\,,\\
\widetilde{W}_{i2}&=\tau_2V_{i2}\tau_2^*.
\end{align*}
Evidently, $(\widetilde{W}_{11}\oplus \widetilde{W}_{12}, \dots , \widetilde{W}_{n1}\oplus \widetilde{W}_{n2})$ is a unitary extension of $(\widetilde{V}_{11}\oplus \widetilde{V}_{12}, \dots , \widetilde{V}_{n1}\oplus \widetilde{V}_{n2})$ and consequently $(\widetilde{W}_{11}\oplus \widetilde{W}_{12}, \dots , \widetilde{W}_{n1}\oplus \widetilde{W}_{n2})$ is a unitary dilation of $(T_1, \dots , T_n)$.

\end{proof}

We conclude this paper with a weaker version of Theorem \ref{Unimain} that assumes conditions $(1)-(4)$ of Theorem \ref{Unimain} for having a unitary dilation of $(T_1, \dots , T_n)$. We have already seen in Theorem \ref{thm:Nagy-Foias} that a c.n.u. tuple of commuting contractions $(T_1, \dots , T_n)$ dilates to commuting unitaries on the minimal unitary dilation space of $T=\prod_{i=1}^n T_i$ with these four conditions. However, because of such weaker hypotheses we will not have a proper converse part.

\begin{thm}\label{coromain}
	Let $T_1,\ldots, T_n \in \mathcal{B}(\mathcal{H})$ be commuting contractions, $T_i'=\prod_{i\neq j} T_j$ for all $1\leq i \leq n$ and $T=\prod_{i=1}^nT_i$. Then $(T_1,\ldots ,T_n)$ possesses a unitary dilation on the minimal unitary dilation space of $T$, if there are projections $P_1,\ldots ,P_n$ and commuting unitaries $U_1,\ldots ,U_n$ in $\mathcal{B}(\mathcal{D}_T)$ such that the following hold for $i=1, \dots, n$:
	\begin{enumerate} 
		\item $D_TT_i=P_i^{\perp}U_i^*D_T+P_iU_i^*D_TT$ ,
		\item  $P_i^{\perp}U_i^*P_j^{\perp}U_j^*=P_j^{\perp}U_j^*P_i^{\perp}U_i^*$ ,
		\item $U_iP_iU_jP_j=U_jP_jU_iP_i$ ,
		\item $D_TU_iP_iU_i^*D_T=D_{T_i}^2$.
	\end{enumerate} 
	Conversely, if $(T_1,\ldots ,T_n)$ possesses a unitary dilation $(\widehat{W}_1,\ldots,\widehat{W}_{n})$ with $W=\prod_{i=1}^nW_i$ being the minimal unitary dilation of $T$, then there are unique projections $P_1,\ldots ,P_n$ and unique commuting unitaries $U_1,\ldots ,U_n$ in $\mathcal{B}(\mathcal{D}_T)$ satisfying the conditions $(1)-(4)$ above.
\end{thm}

\begin{proof}

Let $(T_1, \dots , T_n)$ satisfy the conditions $(1)-(4)$. Suppose $T=U \oplus \widetilde{T}$ is the canonical decomposition of $T$ with respect to $\HS =\HS_1 \oplus \HS_2$, where $\HS_1$ is the maximal reducing subspace of $T$ such that $U=T|_{\HS_1}$ is a unitary and $\widetilde{T}=T|_{\HS_2}$ is a c.n.u. contraction. Then, it is well-known (e.g. see Lemma 2.2 in \cite{Esch-1} or Theorem 3.7 in \cite{S.Pal2}) that $\HS_1, \HS_2$ are common reducing subspaces for $T_1, \dots , T_n$ and $(T_1|_{\HS_1}, \dots , T_n|_{\HS_1})$ is a tuple of commuting unitaries, whereas $(T_1|_{\HS_2}, \dots , T_n|_{\HS_2})$ is a c.n.u. tuple, i.e. a tuple of commuting contractions whose product is a c.n.u. contraction. Since $D_T\equiv D_{\widetilde{T}}$ and $\mathcal D_{\widetilde{T}} \subseteq \HS_2$, we have that the projections $P_1, \dots , P_n$ and commuting unitaries $U_1, \dots , U_n$ satisfy analogues of conditions $(1)\; \& \; (4)$ respectively with $D_T$ and $T_i$ being replaced by $D_{\widetilde{T}}$ and $T_i|_{\HS_2}$ respectively. Now, Theorem \ref{thm:Nagy-Foias} tells us that the c.n.u. tuple $(T_1|_{\HS_2}, \dots , T_n|_{\HS_2})$ can be dilated to commuting unitaries say $(W_1, \dots , W_n)$ on the Sz. Nagy-Foias minimal unitary dilation space for $\widetilde{T}$. It is merely mentioned that there is a unitary $\tau$ from the Sz. Nagy-Foias minimal unitary space to the Sch$\ddot{a}$ffer's minimal unitary dilation space $l^2(\mathcal D_{\widetilde{T}}) \oplus \HS_2 \oplus l^2(\mathcal D_{\widetilde{T^*}})$ of ${\widetilde{T}}$. Clearly $l^2(\mathcal D_{\widetilde{T}}) \oplus \HS_2 \oplus l^2(\mathcal D_{\widetilde{T^*}})$ can be identified with $l^2(\mathcal D_{{T}}) \oplus \HS_2 \oplus l^2(\mathcal D_{{T^*}})$. Thus, $(W_1, \dots , W_n)$ can be identified with a tuple of commuting unitaries say $(\widetilde{W}_1, \dots , \widetilde{W}_n)$ on $l^2(\mathcal D_{{T}}) \oplus l^2(\mathcal D_{{T^*}}) \oplus \HS_2$ dilating $(T_1|_{\HS_2}, \dots , T_n|_{\HS_2})$. It is evident that $(\widetilde{W}_1 \oplus T_1|_{\HS_1}, \dots , \widetilde{W}_n \oplus T_n|_{\HS_1})$ is a unitary dilation of $(T_1, \dots , T_n)$ on the minimal unitary dilation space $l^2(\mathcal D_{{T}}) \oplus l^2(\mathcal D_{{T^*}}) \oplus \HS_2 \oplus \HS_1 \equiv l^2(\mathcal D_{{T}})  \oplus \HS \oplus l^2(\mathcal D_{{T^*}})$ of $T$. The converse part follows from Theorem \ref{Unimain}.

\end{proof}

Note that the class of commuting contractions that dilate to commuting unitaries by satisfying four conditions of Theorem \ref{coromain} is strictly larger than the class satisfying the five conditions of Theorem \ref{Unimain}, though they are being dilated to the same space. In this context we would like to recall Example 5.4 from \cite{Sou:Pra}. Indeed, if we consider
\[
T_1=\begin{bmatrix}
		1&0&0\\0&0&0\\0&0&0
	\end{bmatrix},\; 
	T_2=\begin{bmatrix}
		0&0&0\\0&1&0\\0&0&0
	\end{bmatrix} \text{ and }
	T_3=\begin{bmatrix}
		0&0&0\\0&0&0\\0&0&1
	\end{bmatrix}
	\text{ on } \mathbb{C}^3,
	\]
	we see that the commuting triple $(T_1, T_2 , T_3)$ satisfies conditions $(1)-(4)$ but fails to meet condition-$(5)$ of Theorem \ref{Unimain}. Thus, this triple possesses a unitary dilation $(U_1, U_2,U_3)$ on the minimal unitary dilation space of $T_1T_2T_3$ but the product $U_1U_2U_3$ is not the minimal unitary dilation of $T_1T_2T_3$.
	
	\vspace{0.4cm}
	
	\section{Appendix}
	
	\vspace{0.2cm}
	
	\noindent \textbf{Proof of (\ref{eqn:new-01}).} We have that 
\[
	D_TF_iD_T+D_{T^*}G_i'D_{T^*}T = D_{T_i'}^2T_i+D_{T_i^*}^2T_i'^*T =T_i-T_i'^*T_i'T_i+T_i'^*T-T_iT_i^*T_i'^*T=T_iD_{T}^2.
\]
Applying the relation $TD_T=D_{T^*}T$ we have $ D_TF_iD_T+D_{T^*}G_i'TD_{T}=T_iD_T^2 .$ Thus, we have
$
D_TF_i=T_iD_T-D_{T^*}G_i'T|_{\mathcal{D}_T}.
$ 
Similarly, we can have
$
D_TF_i'D_T+D_{T^*}G_iTD_{T}=T_i'D_T^2,
$ 
which implies $
D_TF_i'=T_i'D_T-D_{T^*}G_iT|_{\mathcal{D}_T}.
$ Now $D_TF_i'F_i'^*D_T=D_{T_i}^2$ leads to the following:
 \begin{align*}
	& (T_i'D_{T}-D_{T^*}G_iT)(D_TT_i'^*-T^*G_i^*D_{T^*})-D_{T_i}^2=0\\
	\implies& T_i'D_{T}^2T_i'^*-T_i'T^*D_{T^*}G_i^*D_{T^*}-D_{T^*}G_iD_{T^*}TT_i'^*+D_{T^*}G_iTT^*G_i^*D_{T^*}-D_{T_i}^2 =0 \\  &\qquad \qquad [\text{ by } T^*D_{T^*}=D_{T}T^*,\, TD_T=D_{T^*}T ]\\
	\implies & T_i'T_i'^*-T_i'T^*TT_i'^*-T_i'T^*D_{T^*}G_i^*D_{T^*}-T_i^*TT_i'^*+T_i'T^*TT_i'^* \\
	& \quad +D_TG_iTT^*G_i^*D_{T^*}-I+T_i^*T_i=0 \quad [\text{ since } D_{T^*}G_iD_{T^*}=D_{T_i'^{*}}^2T_i^*=T_i^*-T_i'T^*]\\
	\implies & D_TG_iTT^*G_i^*D_{T^*}+T_i^*T_i-T_i^*TT_i'^*-T_i'T^*D_{T^*}G_i^*D_{T^*}=I-T_i'T_i'^*\\
	\implies&D_{T^*}G_iTT^*G_i^*D_{T^*}+T_i^*(T_i-TT_i'^*)-T_i'T^*D_{T^*}G_i^*D_{T^*}=D_{T_i'^*}^2\\
	\implies&D_{T^*}G_iTT^*G_i^*D_{T^*}+T_i^*D_{T^*}G_i^*D_{T^*}-T_i'T^*D_{T^*}G_i^*D_{T^*}=D_{T_i'^*}^2\\
	\implies & D_{T^*}G_iTT^*G_i^*D_{T^*}+(T_i^*-T_i'T^*)D_{T^*}G_i^*D_{T^*}=D_{T_i'^*}^2\\
	\implies &D_{T^*}G_iTT^*G_i^*D_{T^*}+D_{T^*}G_iD_{T^*}D_{T^*}G_i^*D_{T^*}=D_{T_i'^*}^2 \\
	\implies &D_{T^*}G_iG_i^*D_{T^*}=D_{T_i'^*}^2.
\end{align*}
Similarly, from $D_TF_iF_i^*D_T=D_{T_i'}^2$ and $D_TF_i=T_iD_T-D_{T^*}G_i'T|_{\mathcal{D}_T}$ we obtain
\begin{equation} \label{eqn:0f3}
D_{T^*}G_i'G_i'^*D_{T^*}=D_{T_i^*}^2.
\end{equation}
As a consequence we have the following. 
\begin{align*}
	D_{T^*}G_iG_i'D_{T^*} =(T_i^*D_{T^*}-D_TF_i'T^*)G_i'D_{T^*}
	&=T_i^*D_{T^*}G_i'D_{T^*}-D_TF_i'F_i'^*D_TT^*\\
	&=T_i^*(T_i'^*-T_iT^*)-D_TF_i'F_i'^*D_TT^*\\
	&=T^*-T_i^*T_iT^*-D_TF_i'F_i'^*D_TT^*\\
	&=D_{T_i}^2T^*-D_TF_i'F_i'^*D_TT^* \\& =0.
\end{align*}
Therefore, $G_iG_i'=0$. Similarly we can prove that $G_i'G_i=0$. Again note that 
\begin{align*}
	D_{T^*}(G_iG_i^*+G_i'^*G_i')D_{T^*}	&=D_{T^*}G_iG_i^*D_{T^*}+D_{T^*}G_i'^*TT^*G_i'D_{T^*}+D_{T^*}G_i'^*D_{T^*}D_{T^*}G_i'D_{T^*}\\
	&=D_{T_i'^*}^2+D_{T^*}TF_i'F_i'^*T^*D_{T^*}+(T_i'-TT_i^*)(T_i'^*-T_iT^*)\\
	&=I-T_i'T_i'^*+TD_{T_i}^2T^*+T_i'T_i'^*-TT^*-TT^*+TT_i^*T_iT^*\\
	&=I+TT^*-TT_i^*T_iT^*-2TT^*+TT_i^*T_iT^*\\
	&=D_{T^*}^2.
\end{align*} 
Similarly, we can prove that $D_{T^*}(G_i^*G_i+G_i'G_i'^*)D_{T^*}=D_{T^*}^2$.  So, we obtain
\[
G_i^*G_i+G_i'G_i'^*=I=G_iG_i^*+G_i'^*G_i'.
\]

\vspace{0.2cm}

\noindent \textbf{Proof of (\ref{claim1}).}	Note that we already have $D_{V^*}V^kh=0$ for all $h\in \mathcal{H}$, $k\in \mathbb{N}$. Therefore, for all $h\in \mathcal{H}$ and $k\in \mathbb{N}$ we have 
\[
	D_{V_i'^*}^2V_i^*V^kh = (I-V_i'V_i'^*)V_i'V^{k-1}h= (V_i'-V_i'V_i'^*V_i')V^{k-1}h=0.
\]
Thus, \eqref{claim1} holds for all vectors in $\overline{span}\{V^kh:k\in \mathbb{N},\;h\in \mathcal{H}  \}$. As $V$ is an isometry, $D_{V^*}^2=D_{V^*}$ and thus \eqref{foruniqueness1} tells us that $ D_{V_i'^*}^2V_i^*=0$ on $\mathcal{N}(\mathcal{D}_T)$ and thus
$
	D_{V_i'^*}^2V_i^*=D_{V^*}D_{V_i'^*}^2V_i^*D_{V^*}.
$ For $h,h'\in \mathcal{H}$, we have that
\begin{align*}
	\langle D_{V^*}XG_iX^*D_{V^*}D_{V^*}h,D_{V^*}h' \rangle = \langle D_{V^*}^2XG_iX^*D_{V^*}^2h,h' \rangle & = \langle D_{V^*}XG_iX^*D_{V^*}h,h' \rangle \\
	&=\langle G_iX^*D_{V^*}h,X^*D_{V^*}h' \rangle\\
	&=\langle G_iD_{T^*}h,D_{T^*}h' \rangle \\
	& \hspace{20mm} [ \text{ since } XD_{T^*}=D_{V^*} \; \& \; X^*X=I]\\
	&=\langle D_{T^*}G_iD_{T^*}h,h' \rangle \\
	&=\langle (T_i^*-T_i'T^*)h,h' \rangle\\
    &=\langle V_i^*h,h' \rangle -\langle V^*h,V_i'^*h' \rangle\\
    &=\langle (V_i^*-V_i'V^*)h,h' \rangle \\
    &=\langle D_{V_i'^*}^2V_i^*h,h' \rangle\\
    &=\langle D_{V^*}D_{V_i'^*}^2V_i^*D_{V^*}h,h' \rangle\\
    &=\langle D_{V_i'^*}^2V_i^*D_{V^*}h,D_{V^*}h' \rangle.
\end{align*}
Hence, the first identity in \eqref{claim1} holds. Similarly, one can prove the other identity.\\

\noindent \textbf{Proof of (\ref{eqn:new-002}).}  For proving $(ii)$ let us first observe the following: \begin{align*}
V_iD_j+D_iE_j=&\begin{bmatrix}
	T_i&0&0&0&\cdots\\
	P_iU_i^*D_T&P_i^{\perp}U_i^*&0&0&\cdots\\ 
	0 &P_iU_i^*&P_i^{\perp}U_i^*&0&\cdots\\
	0 &0 &P_iU_i^*&P_i^{\perp}U_i^*&\cdots\\
	\vdots&\vdots&\vdots&\vdots&\ddots
\end{bmatrix}
\begin{bmatrix}
	D_{T^*}\widetilde{U}_jQ_j&0&0&\cdots\\
	-P_jU_j^*T^*&0&0&\cdots\\
	0&0&0&\cdots\\
	\vdots&\vdots&\vdots&\ddots
\end{bmatrix}\\&+\begin{bmatrix}
D_{T^*}\widetilde{U}_iQ_i&0&0&\cdots\\
-P_iU_i^*T^*&0&0&\cdots\\
0&0&0&\cdots\\
\vdots&\vdots&\vdots&\ddots
\end{bmatrix} \begin{bmatrix}
	\widetilde{U}_jQ_j^{\perp}&\widetilde{U}_jQ_j&0&\cdots\\
	0&\widetilde{U}_jQ_j^{\perp}&\widetilde{U}_jQ_j&\cdots\\
	0&0&\widetilde{U}_jQ_j^{\perp}&\cdots\\
	\vdots&\vdots&\vdots&\ddots
\end{bmatrix}\\
=&\begin{bmatrix}
	T_iD_{T^*}\widetilde{U}_jQ_j+D_{T^*}\widetilde{U}_iQ_i\widetilde{U}_jQ_j^{\perp}&D_{T^*}\widetilde{U}_iQ_i\widetilde{U}_jQ_j&0&\ldots\\
	P_iU_i^*D_TD_{T^*}\widetilde{U}_jQ_j-P_i^{\perp}U_i^*P_jU_j^*T^*-P_iU_i^*T^*\widetilde{U}_jQ_j^{\perp}&P_iU_i^*T^*\widetilde{U}_jQ_j&0&\ldots\\
	-P_iU_i^*P_jU_j^*T^*&0&0&\ldots\\
	0&0&0&\ldots\\
	\vdots&\vdots&\vdots&\ddots
\end{bmatrix} 
\end{align*} and 
\begin{align*}
V_jD_i+D_jE_i=&\begin{bmatrix}
	T_j&0&0&0&\cdots\\
	P_jU_j^*D_T&P_j^{\perp}U_j^*&0&0&\cdots\\ 
	0 &P_jU_j^*&P_j^{\perp}U_j^*&0&\cdots\\
	0 &0 &P_jU_j^*&P_j^{\perp}U_j^*&\cdots\\
	\vdots&\vdots&\vdots&\vdots&\ddots
\end{bmatrix}
\begin{bmatrix}
	D_{T^*}\widetilde{U}_iQ_i&0&0&\cdots\\
	-P_iU_i^*T^*&0&0&\cdots\\
	0&0&0&\cdots\\
	\vdots&\vdots&\vdots&\ddots
\end{bmatrix}\\&+\begin{bmatrix}
	D_{T^*}\widetilde{U}_jQ_j&0&0&\cdots\\
	-P_jU_j^*T^*&0&0&\cdots\\
	0&0&0&\cdots\\
	\vdots&\vdots&\vdots&\ddots
\end{bmatrix} \begin{bmatrix}
	\widetilde{U}_iQ_i^{\perp}&\widetilde{U}_iQ_i&0&\cdots\\
	0&\widetilde{U}_iQ_i^{\perp}&\widetilde{U}_iQ_i&\cdots\\
	0&0&\widetilde{U}_iQ_i^{\perp}&\cdots\\
	\vdots&\vdots&\vdots&\ddots
\end{bmatrix}\\
=&\begin{bmatrix}
	T_jD_{T^*}\widetilde{U}_iQ_i+D_{T^*}\widetilde{U}_jQ_j\widetilde{U}_iQ_i^{\perp}&D_{T^*}\widetilde{U}_jQ_j\widetilde{U}_iQ_i&0&\ldots\\
	P_jU_j^*D_TD_{T^*}\widetilde{U}_iQ_i-P_j^{\perp}U_j^*P_iU_i^*T^*-P_jU_j^*T^*\widetilde{U}_iQ_i^{\perp}&P_jU_j^*T^*\widetilde{U}_iQ_i&0&\ldots\\
	-P_jU_j^*P_iU_i^*T^*&0&0&\ldots\\
	0&0&0&\ldots\\
	\vdots&\vdots&\vdots&\ddots
\end{bmatrix} .
\end{align*}
Clearly $P_iU_i^*P_jU_j^*T^*=P_jU_j^*P_iU_i^*T^*$ and $D_{T^*}\widetilde{U}_iQ_i\widetilde{U}_jQ_j=D_{T^*}\widetilde{U}_jQ_j\widetilde{U}_iQ_i$ follow from $(3)$ and $(3)'$ respectively. Next observe that
	\[
		D_{T^*}TU_iP_iD_{T}=TD_TU_iP_iD_T=T(T_i'-T_i^*T)
		=(T_i'-TT_i^*)T=D_{T^*}Q_i\widetilde{U}_i^*D_{T^*}T=D_{T^*}Q_i\widetilde{U}_i^*TD_{T}.
	\]
	Since both $TU_iP_i$ and $Q_i\widetilde{U}_i^*T$ map $\mathcal{D}_T$ into $\mathcal{D}_{T^*}$, we have that \begin{equation}\label{relbetUiPiQiUitilde}
	TU_iP_i=Q_i\widetilde{U}_i^*T|_{\mathcal{D}_T}.
\end{equation}
Therefore, we have
\[
P_jU_j^*T^*\widetilde{U}_iQ_i=T^*\widetilde{U}_jQ_j\widetilde{U}_iQ_i=T^*\widetilde{U}_iQ_i\widetilde{U}_jQ_j=P_iU_i^*T^*\widetilde{U}_jQ_j
\]
So, for proving $V_iD_j+D_iE_j=V_jD_i+D_jE_i$, it suffices to show 
	\begin{enumerate}
	\item[(a)]	$T_iD_{T^*}\widetilde{U}_jQ_j+D_{T^*}\widetilde{U_i}Q_i\widetilde{U}_{j}Q_j^{\perp}=T_jD_{T^*}\widetilde{U_i}Q_i+D_{T^*}\widetilde{U}_j Q_j\widetilde{U}_{i}Q_i^{\perp} $,
	\item[(b)]$	P_iU_i^*D_TD_{T^*}\widetilde{U}_jQ_j-P_i^{\perp}U_i^*P_jU_j^*T^*-P_iU_i^*T^*\widetilde{U}_jQ_j^{\perp}=P_jU_j^*D_TD_{T^*}\widetilde{U}_iQ_i-P_j^{\perp}U_j^*P_iU_i^*T^*-P_jU_j^*T^*\widetilde{U}_iQ_i^{\perp}
	$.
	\end{enumerate}  For proving $(a)$ we first show that 
\begin{equation}\label{F_iF_j'+F_jF_i'}
	\widetilde{U}_iQ_i^{\perp}\widetilde{U}_jQ_j+\widetilde{U}_iQ_i\widetilde{U}_{j}Q_j^{\perp}=\widetilde{U}_jQ_j^{\perp}\widetilde{U}_iQ_i+\widetilde{U}_jQ_j\widetilde{U}_iQ_i^{\perp}.
\end{equation}
Note that
\begin{align*}
\widetilde{U}_iQ_i^{\perp}\widetilde{U}_jQ_j+\widetilde{U}_iQ_i\widetilde{U}_{j}Q_j^{\perp}&=  \widetilde{U}_i\widetilde{U}_jQ_j-\widetilde{U}_iQ_i\widetilde{U}_jQ_j+\widetilde{U}_iQ_i\widetilde{U}_{j}-\widetilde{U_i}Q_i\widetilde{U}_{j}Q_j\\
&=\widetilde{U}_i\widetilde{U}_j(Q_j+\widetilde{U}_j^*Q_i\widetilde{U}_{j})-2\widetilde{U_i}Q_i\widetilde{U}_{j}Q_j\\
&=\widetilde{U}_j\widetilde{U}_i(Q_i+\widetilde{U}_i^*Q_j\widetilde{U}_{i})-2\widetilde{U_i}Q_i\widetilde{U}_{j}Q_j\hspace{1cm} [\text{From }\eqref{key condition}. ]\\
&=\widetilde{U}_j\widetilde{U}_iQ_i+\widetilde{U}_jQ_j\widetilde{U}_{i} -\widetilde{U}_jQ_j\widetilde{U}_iQ_i-\widetilde{U}_jQ_j\widetilde{U}_{i}Q_i\\
&=\widetilde{U}_jQ_j^{\perp}\widetilde{U}_iQ_i+\widetilde{U}_jQ_j\widetilde{U}_iQ_i^{\perp}.
\end{align*}
 Now we prove $(a)$ using conditions-$(1'), (3')$ and \eqref{F_iF_j'+F_jF_i'} in the following way.
	\begin{align*}
		&T_iD_{T^*}\widetilde{U}_jQ_j+D_{T^*}\widetilde{U}_iQ_i\widetilde{U}_{j}Q_j^{\perp}\\
		=& D_{T^*}\widetilde{U}_iQ_i^{\perp}\widetilde{U}_jQ_j+TD_{T^*}\widetilde{U}_iQ_i\widetilde{U}_jQ_j+D_{T^*}\widetilde{U}_iQ_i\widetilde{U}_{j}Q_j^{\perp}\hspace{1cm} [\text{by condition}-(1')]\\
		=&D_{T^*}(\widetilde{U}_iQ_i^{\perp}\widetilde{U}_jQ_j+\widetilde{U}_iQ_i\widetilde{U}_{j}Q_j^{\perp})+TD_{T^*}\widetilde{U}_iQ_i\widetilde{U}_jQ_j\\
		=&D_{T^*}(\widetilde{U}_jQ_j^{\perp}\widetilde{U}_iQ_i+\widetilde{U}_jQ_j\widetilde{U}_iQ_i^{\perp})+TD_{T^*}\widetilde{U}_j Q_j\widetilde{U}_iQ_i\hspace{1cm} [\text{by }\eqref{F_iF_j'+F_jF_i'} \text{ and condition }-(3')]\\
		=&D_{T^*}\widetilde{U}_jQ_j^{\perp}\widetilde{U}_iQ_i+TD_{T^*}\widetilde{U}_jQ_j\widetilde{U}_iQ_i+D_{T^*}\widetilde{U}_jQ_j\widetilde{U}_iQ_i^{\perp}\\
		=&(D_{T^*}\widetilde{U}_jQ_j^{\perp}+TD_{T^*}\widetilde{U}_j Q_j)\widetilde{U}_iQ_i+D_{T^*}\widetilde{U}_jQ_j\widetilde{U}_iQ_i^{\perp}\\
		=&T_jD_{T^*}\widetilde{U}_i Q_i+D_{T^*}\widetilde{U}_jQ_j\widetilde{U}_{i}Q_i^{\perp}.\hspace{1cm} [\text{by condition}-(1')]
	\end{align*}
	This proves $(a)$. Before proving $(b)$ note that by an argument similar to that in \eqref{F_iF_j'+F_jF_i'}, we can have
\begin{equation} \label{F_i^*F_j'2}
P_iU_i^*P_j^{\perp}U_j^*+P_i^{\perp}U_i^*P_jU_j^*=P_jU_j^*P_i^{\perp}U_i^*+P_j^{\perp}U_j^*P_iU_i^*.
\end{equation} 
Using \eqref{DTUIPiDT}, \eqref{DT*UitildeQiDT*} we have that 
 \begin{align*}
	D_TU_iP_iD_T+D_{T^*}Q_i^{\perp}\widetilde{U}_i^*TD_{T}&=D_{T_i}^2T_i'+D_{T_i'^*}T_i^*T\\
	&=T_i'-T_i^*T+T_i^*T-T_i'T^*T\\
	&=T_i'D_T^2.	
\end{align*}
Hence,
\begin{equation}\label{D_TU_iP_i}
D_TU_iP_i=T_i'D_T-D_{T^*}Q_i^{\perp}\widetilde{U}_i^*T|_{\mathcal{D}_T}.
\end{equation}
Further we observe that  
\begin{align*}
	&(P_iU_i^*D_TD_{T^*}-P_i^{\perp}U_i^*T^*-D_TD_{T^*}\widetilde{U}_iQ_i+T^*\widetilde{U}_iQ_i^{\perp})D_{T^*}\\
	=& P_iU_i^*D_TD_{T^*}^2-P_i^{\perp}U_i^*T^*D_{T^*}-D_TD_{T^*}\widetilde{U}_iQ_iD_{T^*}+T^*\widetilde{U}_iQ_i^{\perp}D_{T^*}\\
	=& P_iU_i^*D_T-P_iU_i^*D_TTT^*-P_i^{\perp}U_i^*D_{T}T^*-D_TD_{T^*}\widetilde{U}_iQ_iD_{T^*}+T^*\widetilde{U}_iQ_i^{\perp}D_{T^*}\hspace{0.5cm}[\text{From  }T^*D_{T^*}=D_TT^*]\\
	=& P_iU_i^*D_T-(P_iU_i^*D_TT+P_i^{\perp}U_i^*D_{T})T^*-D_TD_{T^*}\widetilde{U}_iQ_iD_{T^*}+T^*\widetilde{U}_iQ_i^{\perp}D_{T^*}\\
	=& P_iU_i^*D_T-D_TT_iT^*-D_TD_{T_i^*}^2T_i'^*+T^*\widetilde{U}_iQ_i^{\perp}D_{T^*}\hspace{1cm}[\text{From condition-} (1) \text{ and  }\eqref{DT*UitildeQiDT*}] \\
	=& P_iU_i^*D_T-D_T(T_iT^*+T_i'^*-T_iT_i^*T_i'^*)+T^*\widetilde{U}_iQ_i^{\perp}D_{T^*}\\
	=& P_iU_i^*D_T-D_TT_i'^*+T^*\widetilde{U}_iQ_i^{\perp}D_{T^*}\\
	=& 0 . \hspace{1cm}[\text{From  }\eqref{D_TU_iP_i}]
\end{align*} Therefore, we have 
\begin{equation}\label{lem10}
	P_iU_i^*D_TD_{T^*}-P_i^{\perp}U_i^*T^*|_{\mathcal{D}_{T^*}}=D_TD_{T^*}\widetilde{U}_iQ_i-T^*\widetilde{U}_iQ_i^{\perp}.
\end{equation}
Now we prove $(b)$. We have

\begin{align*}
	& P_iU_i^*D_TD_{T^*}\widetilde{U}_jQ_j-P_i^{\perp}U_i^*P_jU_j^*T^*-P_iU_i^*T^*\widetilde{U}_jQ_j^{\perp}\\
	= & P_iU_i^*(D_TD_{T^*}\widetilde{U}_jQ_j-T^*\widetilde{U}_jQ_j^{\perp})-P_i^{\perp}U_i^*P_jU_j^*T^*\hspace{1cm}[\text{by } \eqref{lem10}]\\
	= & P_iU_i^*(P_jU_j^*D_TD_{T^*}-P_j^{\perp}U_j^*T^*)-P_i^{\perp}U_i^*P_jU_j^*T^*\hspace{1cm}[ \text{by }\eqref{relbetUiPiQiUitilde}]\\
	= & P_iU_i^*P_jU_j^*D_TD_{T^*}-P_iU_i^*P_j^{\perp}U_j^*T^*-P_i^{\perp}U_i^*P_jU_j^*T^*\\
	= & P_jU_j^*P_iU_i^*D_TD_{T^*}-(P_iU_i^*P_j^{\perp}U_j^*+P_i^{\perp}U_i^*P_jU_j^*)T^*\\
	= & P_jU_j^*P_iU_i^*D_TD_{T^*}-(P_jU_j^*P_i^{\perp}U_i^*+P_j^{\perp}U_j^*P_iU_i^*)T^* \hspace{1cm}[\text{by } \eqref{F_i^*F_j'2} ]\\
	= & P_jU_j^*(P_iU_i^*D_TD_{T^*}-P_i^{\perp}U_i^*T^*)-P_j^{\perp}U_j^*P_iU_i^*T^* \\
    = & 	P_jU_j^*(D_TD_{T^*}\widetilde{U}_iQ_i-T^*\widetilde{U}_iQ_i^{\perp})-P_j^{\perp}U_j^*P_iU_i^*T^*\hspace{1cm}[\text{by } \eqref{lem10}]\\
    = & P_jU_j^*D_TD_{T^*}\widetilde{U}_iQ_i-P_j^{\perp}U_j^*P_iU_i^*T^*-P_jU_j^*T^*\widetilde{U}_iQ_i^{\perp}.
\end{align*}	
	
\medskip

\noindent \textbf{Proof of (\ref{eqn:new-003}).}  Note that (\ref{eqn:new-003}) holds if and only if 
	\begin{itemize}
		\item[$(a')$] $V_iV_i^*+D_iD_i^*=I_{\mathcal{K}_0}$ ,
		\item[$(b')$] $D_iE_i^*=0 $ ,
		\item[$(c')$] $E_iE_i^*=I_{l^2(\mathcal{D}_{T^*})}$.
	\end{itemize}
First we observe that the matrix of $E_iE_i^*$ with respect to the decomposition $\mathcal{D}_{T^*}\oplus \mathcal{D}_{T^*}\oplus \cdots$ is \[ \begin{bmatrix}
\widetilde{U_i}Q_i^{\perp}\widetilde{U_i}^*+\widetilde{U_i}Q_i\widetilde{U_i}^*&\widetilde{U_i}Q_iQ_i^{\perp}\widetilde{U_i}^*&0&0&\ldots\\
\widetilde{U_i}Q_i^{\perp}Q_i\widetilde{U_i}^*& \widetilde{U_i}Q_i^{\perp}\widetilde{U_i}^*+\widetilde{U_i}Q_i\widetilde{U_i}^*&\widetilde{U_i}Q_iQ_i^{\perp}\widetilde{U_i}^*&0&\ldots\\
0&\widetilde{U_i}Q_i^{\perp}Q_i\widetilde{U_i}^*& \widetilde{U_i}Q_i^{\perp}\widetilde{U_i}^*+\widetilde{U_i}Q_i\widetilde{U_i}^*&\widetilde{U_i}Q_iQ_i^{\perp}\widetilde{U_i}^*&\ldots\\
0&0&\widetilde{U_i}Q_i^{\perp}Q_i\widetilde{U_i}^*& \widetilde{U_i}Q_i^{\perp}\widetilde{U_i}^*+\widetilde{U_i}Q_i\widetilde{U_i}^*&\ldots\\
\vdots&\vdots&\vdots&\vdots&\ddots
 \end{bmatrix} . \]
Since $\widetilde{U_i}\widetilde{U_i}^*=I$, $Q_i+Q_i^{\perp}=I$ and $Q_i^{\perp}Q_i=0$, we have $(c')$. The $(1,1)$ entry of $D_iE_i^*$ is $D_{T^*}\widetilde{U_i}Q_iQ_i^{\perp}\widetilde{U_i}^*=0$ and the $(2,1)$ entry is $-PU_i^*T^*Q_i^{\perp}\widetilde{U_i}^*$. So, we have from \eqref{relbetUiPiQiUitilde}
\[
-PU_i^*T^*Q_i^{\perp}\widetilde{U_i}^*=-T^*\widetilde{U_i}Q_iQ_i^{\perp}\widetilde{U_i}^*=0.
\]
All other entries of $D_iE_i^*$ are equal to $0$. This proves $(b')$. For proving $(a)$ we first observe that, 
\[
V_iV_i^*= \begin{bmatrix}
	T_iT_i^*&T_iD_TU_iP_i&0&0&\ldots \\
	P_iU_i^*D_TT_i^*&P_iU_i^*D_{T}^2U_iP_i+P_i^{\perp}U_i^*U_iP_i^{\perp}&P_i^{\perp}U_i^*U_iP_i&0&\ldots\\
	0&P_iU_i^*U_iP_i^{\perp}&P_iU_i^*U_iP_i+P_i^{\perp}U_i^*U_iP_i^{\perp}&P_i^{\perp}U_i^*U_iP_i&\ldots\\
	0&0&P_iU_i^*U_iP_i^{\perp}&P_iU_i^*U_iP_i+P_i^{\perp}U_i^*U_iP_i^{\perp}&\ldots\\
	\vdots&\vdots&\vdots&\vdots&\ddots 
\end{bmatrix} .
\]
Using the fact that $U_i^*U_i=I$, $P_i^{\perp}P_i=0$ and $P_i+P_i^{\perp}=I$ we obtain 
\[
V_iV_i^*=\begin{bmatrix}
	T_iT_i^*&T_iD_TU_iP_i&0&0&\ldots \\
	P_iU_i^*D_TT_i^*&I-P_iU_i^*T^*TU_iP_i&0&0&\ldots\\
	0&0&I&0&\ldots\\
	0&0&0&I&\ldots \\
	\vdots&\vdots&\vdots&\vdots&\ddots 
\end{bmatrix}.
\]
Hence matrix of $V_iV_i^*+D_iD_i^*$ with respect to the decomposition $\mathcal{H}\oplus \mathcal{D}_T\oplus \mathcal{D}_T\oplus \cdots $ is 
\[ \begin{bmatrix}
 T_iT_i^*+D_{T^*}\widetilde{U_i}Q_i\widetilde{U_i}^*D_{T^*}&T_iD_TU_iP_i-D_{T^*}\widetilde{U_i}Q_iTU_iP_i&0&0&\ldots \\
 P_iU_i^*D_TT_i^*-P_iU_i^*T^*D_{T^*}\widetilde{U_i}Q_i&I&0&0&\ldots\\
 0&0&I&0&\ldots\\
 0&0&0&I&\ldots \\
 \vdots&\vdots&\vdots&\vdots&\ddots 
 
\end{bmatrix}.
\]
Here the $(1,1)$ entry is equal to the identity by $(4')$. Thus, it remains to prove that
\[
T_iD_TU_iP_i-D_{T^*}\widetilde{U_i}Q_iTU_iP_i|_{\mathcal{D}_{T}}=0.
\]
Note that 
\begin{align*}
	&T_iD_TU_iP_iD_{T}-D_{T^*}\widetilde{U_i}Q_iTU_iP_iD_{T}\\
	&=T_i(T_i'-T_i^*T)-D_{T^*}\widetilde{U_i}Q_i\widetilde{U_i}^*TD_{T}\hspace{1cm}[\text{by }\eqref{relbetUiPiQiUitilde}]\\
	&=T_iT_i'-T_iT_i^*T-D_{T^*}\widetilde{U_i}Q_i\widetilde{U_i}^*D_{T^*}T\\
	&=T-T_iT_i^*T-(I-T_iT_i^*)T\hspace{1cm}[\text{by condition-(4')}]\\
	&=0.
\end{align*}

\medskip

\noindent \textbf{Proof of (\ref{eqn:new-00011}).} Note that
\begin{align*}
&\underline{V_k}D_{k+1}+\underline{D_k}E_{k+1}= \\ &\begin{bmatrix}
	\underline{T_k}D_{T^*}\widetilde{U}_{k+1}Q_{k+1}+D_{T^*}\underline{\widetilde{U}_k}\;\underline{Q_k}\widetilde{U}_{k+1}\;Q_{k+1}^{\perp}&D_{T^*}\underline{\widetilde{U}_k}\;\underline{Q_k}\widetilde{U}_{k+1}\;Q_{k+1}&0&\ldots\\ \underline{P_k}\underline{U_k}^*D_TD_{T^*}\widetilde{U}_{k+1}Q_{k+1}-\underline{P_k}^{\perp}\underline{U_k}^*P_{k+1}U_{k+1}^*T^*-\underline{P_k}\underline{U_k}^*T^*\widetilde{U}_{k+1}\;Q_{k+1}^{\perp}&-\underline{P_k}\underline{U_k}^*T^*\widetilde{U}_{k+1}\;Q_{k+1}  &0&\ldots\\
	-\underline{P_k}\underline{U_k}^*P_{k+1}U_{k+1}^*T^*&0&0&\ldots\\
	0&0&0&\ldots \\
	\vdots&\vdots&\vdots&\ldots
\end{bmatrix}.
\end{align*}
In the above block-matrix the $(3,1)$ entry is equal to $0$. This is because, from \eqref{underlineVk} we have $\underline{P_k}\underline{U_k}^*P_{k+1}U_{k+1}^* =0$. Again, the $(2,2)$ entry is equal to $0$ as we have from \eqref{relbetUiPiQiUitilde} that
\[
\underline{P_k}\underline{U_k}^*T^*\widetilde{U}_{k+1}\;Q_{k+1}= \underline{P_k}\underline{U_k}^*P_{k+1}U_{k+1}^*T^*.
\]
It follows from \eqref{underlineEk} that the $(1,2)$ entry
\[
D_{T^*}\underline{\widetilde{U}_k}\;\underline{Q_k}\widetilde{U}_{k+1}\;Q_{k+1}=0.
\]
Hence $\underline{V_k}D_{k+1}+\underline{D_k}E_{k+1}=\underline{D_{k+1}}$ if and only if
\begin{enumerate}
\item[($a$)] $\underline{T_k}D_{T^*}\widetilde{U}_{k+1}Q_{k+1}+D_{T^*}\underline{\widetilde{U}_k}\;\underline{Q_k}\widetilde{U}_{k+1}\;Q_{k+1}^{\perp}=D_{T^*}\underline{\widetilde{U}_{k+1}}\;\underline{Q_{k+1}}$ (as maps from $\mathcal{D}_{T^*}$ to  $\mathcal{H}$) and
\item[($b$)] $\underline{P_k}\;\underline{U_k}^*D_TD_{T^*}\widetilde{U}_{k+1}\;Q_{k+1}-\underline{P_k}^{\perp}\underline{U_k}^*P_{k+1}U_{k+1}^*T^*-\underline{P_k}\;\underline{U_k}^*T^*\widetilde{U}_{k+1}\;Q_{k+1}^{\perp}=-\underline{P_{k+1}}\;\underline{U_{k+1}}^*T^*$ (as maps from $\mathcal{D}_{T^*}$ to $\mathcal{D}_T$).
\end{enumerate}
First we show that for $m=1,\dots ,n$,
\begin{equation}\label{_1'}
	\underline{T_m}D_{T^*}=D_{T^*}\underline{\widetilde{U}_{m}}\;\underline{Q_m}^{\perp}+TD_{T^*}\underline{\widetilde{U}_{m}}\;\underline{Q_m}.
\end{equation}
We prove this inductively as follows. From condition-$(1')$ we have that \eqref{_1'} holds for $m=1$. Suppose it holds for $m=k$ for some $k\in \mathbb{N}$. We will prove that it holds for $m=k+1$. Note that
\begin{align*}
	\underline{T_{k+1}}D_{T^*}&=T_{k+1}(D_{T^*}\underline{\widetilde{U}_{k}}\;\underline{Q_k}^{\perp}+TD_{T^*}\underline{\widetilde{U}_{k}}\;\underline{Q_k})\\
	&=T_{k+1}D_{T^*}\underline{\widetilde{U}_{k}}\;\underline{Q_k}^{\perp}+T_{k+1}TD_{T^*}\underline{\widetilde{U}_{k}}\;\underline{Q_k}\\
	&=(D_{T^*}\widetilde{U}_{k+1}Q_{k+1}^{\perp}+TD_{T^*}\widetilde{U}_{k+1}Q_{k+1})\underline{\widetilde{U}_{k}}\;\underline{Q_k}^{\perp}+T(D_{T^*}\widetilde{U}_{k+1}Q_{k+1}^{\perp}+TD_{T^*}\widetilde{U}_{k+1}Q_{k+1})\underline{\widetilde{U}_{k}}\;\underline{Q_k}\\
	&\hspace{6cm} [\text{by condition}-(1')]\\
	&=D_{T^*}\widetilde{U}_{k+1}Q_{k+1}^{\perp}\underline{\widetilde{U}_{k}}\;\underline{Q_k}^{\perp}+TD_{T^*}(\widetilde{U}_{k+1}Q_{k+1}\underline{\widetilde{U}_{k}}\;\underline{Q_k}^{\perp}+\widetilde{U}_{k+1}Q_{k+1}^{\perp}\underline{\widetilde{U}_{k}}\;\underline{Q_k})\\
	& \hspace{4cm}+T^2D_{T^*}\widetilde{U}_{k+1}Q_{k+1}\underline{\widetilde{U}_{k}}\;\underline{Q_k} \hspace{1cm}[\text{by }\eqref{underlineEk}, \, \eqref{Ek2}]\\
	&=D_{T^*}\underline{\widetilde{U}_{k+1}}\;\underline{Q_{k+1}}^{\perp}+TD_{T^*}\underline{\widetilde{U}_{k+1}}\;\underline{Q_{k+1}} .
\end{align*} 
Hence by induction \eqref{_1'} holds for all $k=1,\ldots ,n$. 
For proving $(a)$ we first observe that
\begin{align*}
	&\underline{T_k}D_{T^*}\widetilde{U}_{k+1}Q_{k+1}+D_{T^*}\underline{\widetilde{U}_k}\;\underline{Q_k}\widetilde{U}_{k+1}\;Q_{k+1}^{\perp}\\
	&=\underline{T_k}D_{T^*}\widetilde{U}_{k+1}Q_{k+1}-D_{T^*}\underline{\widetilde{U}_{k}}\;\underline{Q_k}^{\perp}\widetilde{U}_{k+1}\;Q_{k+1}+D_{T^*}\underline{\widetilde{U}_{k+1}}\;\underline{Q_{k+1}}\\
	&=(\underline{T_k}D_{T^*}-D_{T^*}\underline{\widetilde{U}_{k}}\;\underline{Q_k}^{\perp})\widetilde{U}_{k+1}\;Q_{k+1}+D_{T^*}\underline{\widetilde{U}_{k+1}}\;\underline{Q_{k+1}}\\
	&=TD_{T^*}\underline{\widetilde{U}_{k}}\;\underline{Q_k}\widetilde{U}_{k+1}\;Q_{k+1}+D_{T^*}\underline{\widetilde{U}_{k+1}}\;\underline{Q_{k+1}}\hspace{1cm} [\text{by }\eqref{_1'}]\\
	&=D_{T^*}\underline{\widetilde{U}_{k+1}}\;\underline{Q_{k+1}}.\hspace{2cm}[\text{by  }\eqref{underlineEk}]
\end{align*}
This proves $(a)$. Now for proving $(b)$ we have
\begin{align*}
	&\underline{P_k}\;\underline{U_k}^*D_TD_{T^*}\widetilde{U}_{k+1}Q_{k+1}-\underline{P_k}^{\perp}\underline{U_k}^*P_{k+1}U_{k+1}^*T^*-\underline{P_k}\underline{U_k}^*T^*\widetilde{U}_{k+1}\;Q_{k+1}^{\perp}\\
	&=\underline{P_k}\;\underline{U_k}^*(D_TD_{T^*}\widetilde{U}_{k+1}Q_{k+1}-T^*\widetilde{U}_{k+1}\;Q_{k+1}^{\perp})-\underline{P_k}^{\perp}\underline{U_k}^*P_{k+1}U_{k+1}^*T^*\\
	&=\underline{P_k}\;\underline{U_k}^*P_{k+1}U_{k+1}^*D_TD_{T^*}-\underline{P_k}\;\underline{U_k}^*P_{k+1}^{\perp}U_{k+1}^*T^*-\underline{P_k}^{\perp}\underline{U_k}^*P_{k+1}U_{k+1}^*T^*\hspace{1cm}[\text{by }\eqref{lem10}]\\
	&=-\underline{P_{k+1}}\;\underline{U_{k+1}}^*T^*.\hspace{2cm}[\text{by }\eqref{underlineVk}]
\end{align*}
This proves $(b)$.

\vspace{0.4cm}

\section{Data availability statement}

 \begin{enumerate}
 
 \item Data sharing is not applicable to this article as no datasets were generated or analysed during
the current study.\\

\item In case any datasets are generated during and/or analysed during the current study, they must
be available from the corresponding author on reasonable request.

\end{enumerate}

\end{document}